\documentclass[11pt]{amsart}
\usepackage[textwidth=13.3cm,textheight=22.3cm,bindingoffset=4mm]{geometry}
\usepackage{graphicx,comment}
\usepackage[colorlinks]{hyperref}
\usepackage[mathscr,mathcal]{euscript}
\usepackage{microtype}

\usepackage{enumerate}
\usepackage[latin1]{inputenc}
\usepackage{dsfont}
\usepackage{amssymb,amsthm,amsmath}

\usepackage{mathrsfs}

\theoremstyle{plain}
\newtheorem{prop}{Proposition}[section]
\newtheorem{lemma}[prop]{Lemma}

\newtheorem{thm}[prop]{Theorem}
\newtheorem*{thmAnn}{Theorem A}
\newtheorem*{thmBnn}{Theorem B}

\newtheorem{cor}[prop]{Corollary}
\newtheorem*{cornn}{Corollary}

\theoremstyle{definition}

\theoremstyle{remark}

\newcommand{\open}{\mathcal O}
\newcommand{\ok}{\mathcal E}

\newcommand{\homsp}{\mathcal X}

\DeclareMathOperator{\diam}{diam}

\DeclareMathOperator{\End}{End}

\DeclareMathOperator{\SL}{SL}

\DeclareMathOperator{\supp}{supp}

\DeclareMathOperator{\Ad}{Ad}
\DeclareMathOperator{\ad}{ad}

\DeclareMathOperator{\height}{ht}

\newcommand\N{\mathbb{N}}

\newcommand\R{\mathbb{R}}
\newcommand\Z{\mathbb{Z}}

\newcommand{\mc}[1]{\mathcal #1}
\newcommand{\mf}[1]{\mathfrak #1}

\newcommand{\mft}[2]{\mathfrak #1\mathfrak #2}
\newcommand{\wt}{\widetilde}

\newcommand{\eps}{\varepsilon}

\DeclareMathOperator{\id}{id}

\newcommand{\sceq}{\mathrel{\mathop:}=}

\begin{document}

\title[Escape of mass and entropy for diagonal flows]{Escape of mass and entropy for diagonal flows in real rank one situations}
\author{M.\@ Einsiedler}
\author{S.\@ Kadyrov}
\author{A.\@ Pohl}
\address[ME]{Departement Mathematik, ETH Z\"urich, R\"amistrasse 101, 8092 Z\"urich, Switzerland}
\address[SK]{School of Mathematics, University of Bristol, Bristol, UK}
\address[AP]{Mathematisches Institut, Georg-August-Universit\"at G\"ottingen,  Bunsenstr. 3-5, 37073 G\"ottingen}
\email[ME]{manfred.einsiedler@math.ethz.ch}
\email[SK]{shirali.kadyrov@bristol.ac.uk}
\email[AP]{pohl@uni-math.gwdg.de}
\thanks{M.E.\@ acknowledges the support by the SNF (Grant 200021-127145). S.K.\@ acknowledges the support by the EPSRC. A.P.\@ acknowledges the support by the SNF (Grant 200021-127145) and the Volkswagen Foundation}

\begin{abstract}
Let $G$ be a connected semisimple Lie group of real rank $1$ with finite center, let $\Gamma$ be a non-uniform lattice in $G$ and $a$ any diagonalizable element in $G$. We investigate the relation between the metric entropy of $a$ acting on the homogeneous space $\Gamma\backslash G$ and escape of mass. Moreover, we provide bounds on the escaping mass and, as an application, we show that the Hausdorff dimension of the set of orbits (under iteration of $a$) which miss a fixed open set is not full.
\end{abstract}

\maketitle
\tableofcontents

\section{Introduction}

Let $G$ be a connected semisimple ({real}) Lie group of $\R$-rank $1$ with finite center and $\Gamma$ a lattice in $G$. Suppose that
\[
 \homsp \sceq \Gamma\backslash G
\]
denotes the associated homogeneous space. Let $A$ be a one-parameter subgroup consisting of $\R$-diagonalizable elements. Pick an element $\tilde a\in A\smallsetminus \{\id\}$ and consider the right action 
\[
T \colon\left\{
\begin{array}{ccl}
\homsp & \to & \homsp
\\
x & \mapsto &  x\tilde a
\end{array}
\right.
\]
of $\tilde a$ on $\homsp$. Further let $(\mu_n)_{n\in\N}$ be a sequence of $T$-invariant probability measures on $\homsp$ which converges in the weak* topology to the measure $\nu$. 

If $\nu$ is itself a probability measure (which is always the case if $\Gamma$ is cocompact), then upper semi-continuity of metric entropy is well-known, that is
\[
 \limsup_{n\to\infty} h_{\mu_n}(T) \leq h_{\nu}(T).
\]
In this article we investigate the case that $\Gamma$ is non-cocompact and $\nu$ is not a probability measure. We show that if upper semi-continuity does not hold, the amount by which it fails is controlled by the escaping mass. More precisely, the main result can be stated as follows.

\begin{thmAnn}
Let $h_m(T)$ denote the maximal metric entropy of $T$ and suppose that $\nu(\homsp) > 0$. Then 
\[
 \nu(\homsp) h_{\frac{\nu}{\nu(\homsp)}}(T) + \tfrac12 h_m(T) \cdot \left( 1-\nu(\homsp) \right) \geq \limsup_{n\to\infty} h_{\mu_n}(T).
\]
\end{thmAnn}

In \cite{Kadyrov_Pohl} it is shown that the factor $\tfrac12$ is sharp. A consequence of this theorem is the following result about escape of mass, which is of interest on its own.

\begin{cornn}
Suppose that $\limsup h_{\mu_n}(T) \geq c$. Then 
\[
 \nu(\homsp) \geq  \frac{2c}{h_m(T)} -1.
\]
\end{cornn}

Thus, if the entropy on the sequence  $(\mu_n)$ is high, meaning at least $\tfrac12h_m(T) + \eps$, then not all of the mass can escape and the remaining mass can be bounded quantitively.

For $\homsp=\SL_2(\Z)\backslash \SL_2(\R)$ and $T$ being the time-one map this control on escape of mass is already shown in \cite{ELMV}. For recent  results of this kind in different settings and their applications we refer to \cite{Einsiedler_Kadyrov, Kadyrov, KKLM}. 

In case of equality in the corollary above, Theorem~A yields the following consequence for the remaining normalized measure.

\begin{cornn}
If $\limsup h_{\mu_n}(T) \geq c$ and 
\[
 \nu(\homsp) =  \frac{2c}{h_m(T)} -1 > 0,
\]
then $h_{\frac{\nu}{\nu(\homsp)}}(T)=h_m(T)$ and~$\frac\nu{\nu(\homsp)}$
is the Haar measure on~$\homsp$.
\end{cornn}

As an application of these results and the methods for their proofs we show in Section~\ref{notfull} the following observation, 
thereby answering a question of Barak Weiss. Its positive solution is already used in \cite{Hubert_Weiss}.

\begin{thmBnn}
Let $\mc O$ be an open nonempty subset of $\homsp$, and let $\ok$ be the set of points in $\homsp$ whose forward trajectories (forward $A$-orbits) do not intersect $\mc O$. Then the Hausdorff dimension of $\ok$ is strictly smaller than the (Hausdorff) dimension of $\homsp$.
\end{thmBnn}

We outline the strategy of proof for Theorem~A. The key tool for its proof is the existence of a finite partition $\eta$ of $\homsp$ such that for each $T$-invariant probability measure $\mu$ on $\homsp$ the entropy of $\mu$, the entropy of the partition $\eta$ and the mass ``high'' in the cusps of $\homsp$ are seen to be related as in Theorem~A. More precisely, if $\homsp_{>s}$ denotes the part of $\homsp$ above height $s$ (the notion of height is defined in Section~\ref{sec_height} below), then 
\[
 h_{\mu}(T) \leq h_{\mu}(T,\eta) + c_s + \tfrac12 h_m(T) \mu(\homsp_{>s})
\]
with a global constant $c_s$ such that $c_s\to 0$ as $s\to\infty$. We remark that $\eta$ is independent of $\mu$. To achieve this we use a partition of $\homsp$ into a fixed compact part, the part $\homsp_{>s}$ above height $s$, and the strip between the compact part and $\homsp_{>s}$. The compact part is refined into very small sets, depending on the width of the strip, such that this part and the strip do not contribute to entropy. 

The entropy of $\mu$ is estimated from above using the Brin-Katok Lemma, which reduces this task to counting Bowen balls needed to cover some set of fixed positive measure. In Lemma~\ref{coverimp} below we provide a non-trivial bound for this number. In order to be able to establish this result, we translate the situation to Siegel sets in $G$ (which is possible thanks to a result of Garland and Raghunathan \cite{Garland_Raghunathan} on fundamental domains), and conduct a detailed study how nearby trajectories behave high up in the cusp. 

These investigations do not use the classification of $\R$-rank $1$ simple Lie groups. Rather we take advantage of the uniform and easy to manipulate construction of rank $1$ symmetric spaces of noncompact type provided by \cite{CDKR1} and \cite{CDKR2} and the coordinate system of the associated Lie groups adapted to their geometry.

\textit{Acknowledgment.} We thank the anonymous referees for  many valuable comments that helped to improve the presentation of the paper.

\section{Fundamental domains in the cusps}\label{sec_funddom}

Let $A$ be a one-parameter~$\R$-diagonalizable subgroup in $G$ containing the diagonalizable element $\tilde a$ defining the transformation $T$ on $\homsp$ via $x\mapsto x\tilde a$. Let $C = C_A(G)$ denote the centralizer of $A$ in $G$ and let $\mf c$ be its Lie algebra.
Let $\mf g$ denote the Lie algebra of $G$. Since $G$ is of $\R$-rank $1$, there exists a group homomorphism $\alpha\colon A \to (\R_{>0}, \cdot)$
such that with
\[
 \mf g_j \sceq \left\{ X \in \mf g \left\vert\ \forall\, a\in A\colon \Ad_aX = \alpha(a)^{\frac{j}{2}} X \right.\right\}, \quad j \in \{\pm 1, \pm 2\},
\]
we have the direct sum decomposition
\begin{equation}\label{rootspacedecomp}
 \mf g = \mf g_{-2} \oplus \mf g_{-1} \oplus \mf c \oplus \mf g_1 \oplus \mf g_2.
\end{equation}
We choose the homomorphism $\alpha$ such that $\alpha(\tilde a)>1$. The Lie algebra $\mf g$ is the direct product of a simple Lie algebra and a compact one. Unless this simple Lie algebra is isomorphic to $\mft so(1,n)$,  the homomorphism $\alpha$ is then unique and \eqref{rootspacedecomp} is the restricted root space decomposition of $\mf g$. If the simple factor of $\mf g$ is isomorphic to $\mft so(1,n)$ for some $n\in \N$, $n\geq 2$, then there are two choices for $\alpha$.  Depending on the choice, either $\mf g_2$ or $\mf g_1$ is trivial. In this case  \eqref{rootspacedecomp} simplifies to
\[
 \mf g = \mf g_{-1} \oplus \mf c \oplus \mf g_1 \quad\text{resp.}\quad \mf g = \mf g_{-2} \oplus \mf c \oplus \mf g_2,
\]
each of which is the restricted root space decomposition of $\mf g$. The first one corresponds to the Cayley-Klein models of real hyperbolic spaces, the second one to the Poincar\'{e} models. Define $\mf n \sceq \mf g_2 \oplus \mf g_1$ and let $N$ be the connected, simply connected Lie subgroup of $G$ with Lie algebra $\mf n$. By the theorem concerning
Iwasawa decompositions of~$G$, there exists a maximal compact subgroup $K$ of $G$ such that 
\[
 N \times A \times K \to G,\quad (n,a,k) \mapsto nak
\]
is a diffeomorphism. Let 
\[
 M \sceq K \cap C.
\]
For any $s>0$ we set
\[
 A_s \sceq \{ a \in A \mid \alpha(a) > s \}.
\]
Moreover, for any $s>0$ and any compact subset $\eta$ of $N$ we define the Siegel set
\[
 \Omega(s,\eta) \sceq \eta A_s K.
\]
Garland and Raghunathan provide the following result on fundamental domains for the non-cocompact lattice $\Gamma$ in $G$.

\begin{prop}[Theorem~0.6 and 0.7 in \cite{Garland_Raghunathan}]\label{funddomGR}
There exists $s_0 > 0$, a compact subset $\eta_0$ of $N$ and a finite subset $\Xi$ of $G$ such that 
\begin{enumerate}[{\rm (i)}]
\item\label{funddomGRi} $G=\Gamma\Xi\Omega(s_0,\eta_0)$,
\item for all $\xi\in\Xi$, the group $\Gamma\cap \xi N \xi^{-1}$ is a cocompact lattice in $\xi N \xi^{-1}$,
\item for all compact subsets $\eta$ of $N$ the set
\[
 \{ \gamma \in \Gamma \mid \gamma\Xi\Omega(s_0,\eta) \cap \Omega(s_0,\eta) \not=\emptyset \}
\]
is finite,
\item\label{funddomGRiv} for each compact subset $\eta$ of $N$ containing $\eta_0$, there exists $s_1>s_0$ such that for all $\xi_1,\xi_2\in\Xi$ and all $\gamma\in\Gamma$ with $\gamma\xi_1 \Omega(s_0,\eta) \cap \xi_2\Omega(s_1,\eta) \not= \emptyset$ we have $\xi_1 = \xi_2$ and $\gamma\in \xi_1 NM \xi_1^{-1}$.
\end{enumerate}
\end{prop}

For the remainder of this article we fix $s_1 > s_0 > 0$, a compact subset $\eta_0$ of $N$ and a finite subset $\Xi$ of $G$ which satisfy \eqref{funddomGRi}-\eqref{funddomGRiv} of Proposition~\ref{funddomGR} with $\eta \sceq \eta_0$.

The elements of $\Xi$ are a minimal set of representatives for the cusps of
\[
 \homsp \sceq \Gamma\backslash G,
\]
and for each $\xi\in\Xi$, the Siegel set $\xi\Omega(s_1,\eta)$ modulo $\Gamma\cap \xi NM \xi^{-1}$ is a neighborhood of the corresponding cusp of $\homsp$. In the following we will often identify this cusp with its neighborhood  $(\Gamma\cap \xi NM\xi^{-1})\backslash \xi \Omega(s_1,\eta) \subseteq \homsp$, and also refer to the latter one as the cusp represented by $\xi$.

\section{The height function}\label{sec_height}

For each $\xi\in\Xi$, we introduce a height function which measures how far a point $x\in\homsp$ is ``in the cusp represented by $\xi$''. More precisely, the $\xi$-height of $x$ is the maximal value $\alpha(a)$ for an $x$-representative $\xi nak$ in $G=\xi NAK$. The maximum over all $\xi$-heights gives the total height of $x\in\homsp$. For a coordinate-free definition of the height functions, we introduce a representation derived from the adjoint representation. This representation was also used in \cite{Dani}.

For each $\xi\in\Xi$ we set
\[
 L_\xi \sceq \xi NM \xi^{-1}
\]
and denote its Lie algebra by $\mf l_\xi$. Set $\ell \sceq \dim \mf l_\xi$ (which in fact is independent of $\xi$) and let $V$ be the $\ell$-th exterior power of $\mf g$,  
\[
  V \sceq \bigwedge\nolimits^{\!\!\ell} \mf g.
\]
Let $\varrho$ be the right\footnote{When applying $\varrho(g)$ for $g\in G$ to $v\in V$ we will write $v\varrho(g)$ instead of $\varrho(g)v$ to stress that it is a right action.} $G$-action on $V$ given by the $\ell$-th exterior power of 
\[
\Ad \circ\ (\cdot)^{-1} \colon G \to \End(\mf g), \quad g \mapsto \Ad_{g^{-1}},
\]
hence
\[
 \varrho \sceq \bigwedge\nolimits^{\!\!\ell} \big( \Ad \circ\ (\cdot)^{-1} \big) \colon G \to \End(V).
\]
We fix a non-zero element $v_\xi$ in the one-dimensional space
\[
 W_\xi \sceq \bigwedge\nolimits^{\!\!\ell} \mf l_\xi
\]
and let
\[
 \theta_\xi \colon \xi NMA \xi^{-1} \to \R_{>0}
\]
be the unique group homomorphism into the multiplicative group $(\R_{>0},\cdot)$ such that for all $g\in \xi NMA \xi^{-1}$ we have
\[
 v_\xi \varrho(g) = \theta_\xi(g) v_\xi.
\]
One easily shows that  $\theta_\xi(g) =1$ for $g$ in the connected component of $L_\xi$, and
\[
 \theta_\xi(\xi a \xi^{-1}) = \alpha(a)^{-\left(\frac12\dim \mf g_1 +  \dim \mf g_2\right)}
\]
for $a\in A$. Let
\[
 q \sceq \tfrac12\dim \mf g_1 +  \dim \mf g_2.
\]
We choose a $\varrho(K)$-invariant inner product $\langle \cdot, \cdot \rangle$ on $V$ (e.g.\@ induced by the Killing form) and denote its associated norm by $\|\cdot\|$. 

For $\xi\in\Xi$, the \textit{$\xi$-height} of $x\in\homsp$ is defined as
\begin{equation}\label{xiheight}
 \height_\xi(x) \sceq \sup\left\{ \left( \frac{ \| v_\xi \varrho(g) \| }{ \| v_\xi \varrho(\xi) \| }\right)^{-\frac1q} \left\vert\ g\in G,\ x = \Gamma g  \vphantom{\left( \frac{ \| v_\xi \varrho(\gamma g) \| }{ \| v_\xi \varrho(\xi) \| }\right)^{\frac1q}  } \right.\right\}.
\end{equation}
If $g\in G$ is represented as $g=\xi n a k$ with $n\in N$, $a\in A$ and $k\in K$, then by definition
\[
\left( \frac{ \| v_\xi \varrho(g) \| }{ \| v_\xi \varrho(\xi) \| }\right)^{-\frac1q} = \alpha(a).
\]
Hence this value only depends on the $A$-components of $g$ when represented in $\xi NAK (=G)$, of which we may think as an Iwasawa decomposition of $G$ relative to $\xi$.

The \textit{height} of $x\in\homsp$ is
\[
 \height(x) \sceq \max\big\{ \height_\xi(x) \ \big\vert\  \xi\in\Xi \big\}.
\]
For $s>0$ and $\xi\in\Xi$ we set
\[
 \homsp(\xi,s) \sceq \{x\in\homsp\mid \height_\xi(x) > s\}
\]
and
\begin{equation}\label{overs}
 \homsp_{>s} \sceq \{x\in\homsp\mid \height(x) > s\} = \bigcup_{\xi\in\Xi} \homsp(\xi, s).
\end{equation}

In the following we will see that the points in $\homsp(\xi,s)$ correspond to the elements in the Siegel set $\xi\Omega(s,\eta)$. To that end let $B_\delta$ denote the open $\|\cdot\|$-ball in $V$ with radius $\delta > 0$, centered at $0$. We define 
\[
 \delta_\xi(s) \sceq  s^{-q} \| v_\xi \varrho(\xi) \|.
\]

\begin{prop}[Corollary~2.3 in \cite{Dani}]\label{samenorm}
Let $\xi \in \Xi$, $s>0$, and $g\in G$. Then $\Gamma g \in \Gamma\backslash \Gamma\xi\Omega(s,\eta)$ if and only if $v_\xi \varrho(\gamma g) \in B_{\delta_\xi(s)}$ for some $\gamma\in\Gamma$. Further, if $s \geq s_1$ and $\gamma_1,\gamma_2\in\Gamma$ satisfy $v_\xi \varrho(\gamma_jg) \in B_{\delta_\xi(s)}$ for $j=1,2$, then $v_\xi\varrho(\gamma_1 g) \in \{ \pm v_\xi \varrho(\gamma_2 g)\}$.
\end{prop}

Thus
\[
 \homsp(\xi,s) = \Gamma\backslash \Gamma\xi\Omega(s,\eta)
\]
for all $\xi\in\Xi$ and $s>0$. If $s\geq s_1$, the supremum in \eqref{xiheight} is attained. Moreover, by Proposition~\ref{funddomGR}\eqref{funddomGRiv},  
\[
 \homsp(\xi, s)\cap \homsp(\xi',s) =\emptyset
\]
if $\xi\not=\xi'\in \Xi$. Hence the sets $\homsp(\xi, s)$ are then disjoint neighborhoods of the cusps of $\homsp$, and the union in \eqref{overs} is disjoint.

\section{Coordinate system for $G$}

Recall that the Lie algebra $\mf g$ is the direct sum of a simple Lie algebra of rank $1$ and a compact one. Since the height function is right-$\varrho(K)$-invariant and all further considerations are right-$\varrho(K)$-invariant, we can restrict to $\mf g$ being simple. \cite{CDKR1} and \cite{CDKR2} provide a classification-free construction of all Riemannian symmetric spaces of noncompact type and rank one. Their results rely on the choice of a certain coordinate system for real simple Lie groups $G$ of real rank $1$, which allows us to treat all these groups without refering to their classification. In the following we recall this coordinate system, the one for the associated symmetric spaces and some essential formulas.

The semidirect product $NA$ is parametrized by 
\[
 \R_{>0} \times \mf g_2 \times \mf g_1 \to NA,\quad (s,Z,X) \mapsto \exp(Z+X)\cdot a_s,
\]
where we may assume $s\sceq \alpha(a_s)$. The (left) action of $a_s=(s,0,0)\in A$ on $n=(1,Z,X)\in N$ is then given by
\[
 a_sn = (s, sZ, s^{1/2} X).
\]
We define an inner product on $\mf n = \mf g_2 \oplus\mf g_1$ as follows. Let $\mf k$ be the Lie algebra of $K$. Let $\theta$ be a Cartan involution of $\mf g$ such that $\mf k$ is its $1$-eigenspace.  For $X,Y\in\mf n$ we define
\[
 \langle X, Y \rangle \sceq -\frac{1}{\dim\mf g_1 +4 \dim\mf g_2} B(X,\theta Y)
\]
where $B$ is the Killing form of $\mf g$. It is well-known that $\langle\cdot,\cdot\rangle$ is an inner product on $\mf n$.  As in \cite{CDKR1, CDKR2}, we identify $G/K \cong NA \cong \R_{>0}\times\mf g_2 \times \mf g_1$ with 
\[
 D \sceq \left\{ (t,Z,X)_D \in \R\times \mf g_2 \times \mf g_1 \left\vert\ t>\tfrac14|X|^2 \right.\right\}
\]
via
\[
\R_{>0}\times\mf g_2\times\mf g_1  \to  D,\quad (t,Z,X)  \mapsto  (t+\tfrac14|X|^2, Z, X)_D.
\]
We will include the subscript~$D$ when denoting elements~$(\cdot,\cdot,\cdot)_D$  of the symmetric space~$D$
to avoid confusion with elements of the group~$NA$.
The (left) action of an element $s = (t_s,Z_s,X_s)\in NA$ on a point $p=(t_p,Z_p,X_p)_D \in D$ becomes
\[
 s.p = \big(t_st_p + \tfrac14|X_s|^2 + \tfrac12t_s^{1/2} \langle X_s,X_p\rangle, Z_s + t_sZ_p + \tfrac12t_s^{1/2}[X_s,X_p], X_s+t_s^{1/2}X_p\big)_D.
\]
These coordinates of $G/K$ enable us to use \cite{CDKR1, CDKR2}, and they simplify some of the expressions below, in particular the one for the geodesic inversion. To state the geodesic inversion, we define the linear map 
\[
J \colon \mf g_2 \to \End(\mf g_1), \quad Z \mapsto J_Z,
\]
via
\[
 \langle J_ZX, Y\rangle = \langle Z, [X,Y]\rangle \quad\text{for all $X,Y\in \mf g_1$.}
\]
Then the geodesic inversion $\sigma$ of $D$ at $o\sceq (1,0,0)_D$ is given by (see \cite{CDKR2})
\[
 \sigma(t,Z,X)_D = \frac{1}{t^2 + |Z|^2} \big(t, -Z, (-t+J_Z)X\big)_D.
\]
We identify $\sigma$ with an element in $K$ which acts as geodesic inversion {on~$D=G/K$ at~$o$}. Then $G$ has the Bruhat decomposition (\cite[Theorem~6.4]{CDKR2})
\[
 G = NAM \cup NAM\sigma N.
\]
Multiplying this with $\xi\in\Xi$ from the left and $\sigma$ from the right, we get
\[
 G =  \xi NAM\sigma \cup \xi NAMU
\]
with $U\sceq \sigma N \sigma$. This decomposition provides a coordinate system on $G$ adapted to the cusp represented by $\xi$. The set $\xi NAM\sigma$ we call the \textit{small $\xi$-Bruhat cell} and $\xi NAMU$ the \textit{big $\xi$-Bruhat cell}.

The group $M$ is parametrized by the pairs $(\varphi,\psi)$ consisting of the orthogonal endomorphisms $\varphi$ on $\mf g_2$ resp.\@ $\psi$ on $\mf g_1$ such that $\psi(J_ZX) = J_{\varphi(Z)}\psi(X)$ for all $(Z,X)\in\mf g_2\times\mf g_1$. The action of $(\varphi,\psi)\in M$ on $p=(t,Z,X)_D\in D$ is given by 
\[
 (\varphi,\psi).p=(t,\varphi(Z),\psi(X))_D.
\]
By \cite[Proposition~7.1]{CDKR2}, $|J_ZX|=|Z||X|$ for all $Z\in\mf g_2$, $X\in\mf g_1$.

\section{Variation of height}

Suppose that the point $x\in \homsp$ is of big height and its trajectory stays far out for some time. In this section, we provide non-trivial bounds on the unstable components of a group element $g\in G$ representing $x$. In Proposition~\ref{Vbound2} below, this bound implies constraints on the perturbation allowed for $x$ without destroying the qualitative behavior of its trajectory during this time.

\begin{lemma}\label{sigmanormal}
Let $a_t,a_r\in A$, $m\in M$ and $n\in N$ with $n=(1,Z,X)$  such that $\sigma m n a_t \in N a_r K$. Then 
\[
  r = \frac{t}{\left(t+\frac14|X|^2\right)^2 + |Z|^2}.
\]
\end{lemma}

\begin{proof}
By Iwasawa decomposition we know that $\sigma m n a_t = n' a_r k$ for suitable $n'\in N$, $k\in K$ and $r\in\R_{>0}$.
Suppose that $m=(\varphi,\psi)$. Applying both $\sigma m n a_t$ and $n' a_r k$ to the base point $o=(1,0,0)_D$ in $D$, we find
\[
 \sigma m n a_t \cdot o = n' a_r k \cdot o = n' a_r \cdot o.
\]
In the coordinates of $D$ one easily calculates that 
\begin{align*}
 \sigma m n a_t &\cdot o  = 
\\
& = \frac{1}{\left( t + \frac14|X|^2\right)^2 + |Z|^2}
\left( t+ \tfrac14|X|^2, -\varphi(Z), \left(-t - \tfrac14|X|^2 + J_{\varphi(Z)}\right)\psi(X) \right)_D.
\end{align*}
Suppose that $n'=(1,Z',X')$. Then 
\[
 n' a_r \cdot o = \left( r + \tfrac14|X'|^2, Z', X'\right)_D.
\]
Thus
\begin{align*}
 X' &= \frac{1}{\left(t+\frac14|X|^2\right)^2 + |Z|^2}\left( -t -\tfrac14|X|^2 + J_{\varphi(Z)}\right) \psi(X),
\\
|X'|^2 & = \frac{1}{\big(\big( t+ \frac14|X|^2\big)^2 + |Z|^2 \big)^2} \Big( \left( t+ \tfrac14|X|^2\right)^2|X|^2 + 
\left|J_{\varphi(Z)}\psi(X) \right|^2 
\\
& \hphantom{\frac{1}{\left( \left( t+ \frac14|X|^2\right)^2 + |Z|^2 \right)^2} \Big(  t+\frac14 }- 2 \left( t+ \tfrac14|X|^2\right)\left\langle \psi(X), J_{\varphi(Z)}\psi(X)\right\rangle\Big)
\\
& =\frac{|X|^2}{\left( t+ \frac14|X|^2\right)^2 + |Z|^2},
\intertext{and}
 r & = \frac{t+\frac14|X|^2}{\left( t+ \frac14|X|^2\right)^2 + |Z|^2} - \tfrac14|X'|^2 = \frac{t}{\left(t+\frac14|X|^2\right)^2 + |Z|^2}.
\end{align*}
\end{proof}

\begin{lemma}\label{value}
Let $\xi\in \Xi$ and $g\in G$. If $g=\xi n a_s m \sigma$ with $n\in N$ and $m\in M$, then 
\[
 \left( \frac{\|v_\xi\varrho(ga_t)\|}{\|v_\xi\varrho(\xi)\|}\right)^{-\frac1q} = \frac{s}{t}.
\]
If $g= \xi n a_s m \sigma (1,Z,X)\sigma$ with $n\in N$ and $m\in M$, then
\[
 \left(\frac{\|v_\xi\varrho(ga_t)\|}{\|v_\xi\varrho(\xi)\|}\right)^{-\frac1q} = s\cdot\frac{ \frac1t }{ \left(\frac1t+\frac14|X|^2\right)^2 + |Z|^2 }.
\]
\end{lemma}

Recall the identification of the cusp represented by $\xi\in\Xi$ with the cusp neighborhood $(\Gamma\cap \xi NM\xi^{-1})\backslash \xi \Omega(s_1,\eta)$ from Section~\ref{sec_funddom}. Let us note that the first case corresponds to a trajectory pointing straight out of the cusp represented by $\xi$.  In the second case, the element $u=\sigma (1,Z,X)\sigma$ determines the perturbation to the trajectory pointing straight into the cusp. If $(Z,X) = (0,0)$, the second case correspond to a trajectory pointing straight into the cusp, and the formula simplifies to 
\[
  \left(\frac{\|v_\xi\varrho(ga_t)\|}{\|v_\xi\varrho(\xi)\|}\right)^{-\frac1q} = st.
\]

\begin{proof}[Proof of Lemma~\ref{value}]
At first we suppose that $g= \xi n  a_s m \sigma$. Then 
\[
 ga_t = \xi n a _{s/t} m \sigma = \xi na_{s/t}\xi^{-1} \xi m\sigma \quad\in \big(\xi NA\xi^{-1}\big) \big(\xi K\big).
\]
Hence
\begin{align*}
\| v_\xi \varrho(g a_t)\|  & = \theta_\xi(\xi a_{s/t} \xi^{-1}) \| v_\xi \varrho(\xi) \| = \left(\frac{s}{t}\right)^{-q} \| v_\xi \varrho(\xi) \|.
\end{align*}
Suppose now that $g=\xi n m a_s u$ with $u=\sigma n' \sigma$ and $n'=(1,Z,X)$. Then (for some $m'\in M$)
\begin{align*}
\| v_\xi \varrho(ga_t) \| & = s^{-q} \| v_\xi \varrho(\xi \sigma m' n'\sigma a_t) \| = s^{-q} \| v_\xi \varrho(\xi \sigma m' n' a_{1/t}\sigma)\|
\\
&  = s^{-q} \| v_\xi \varrho(\xi \sigma m' n' a_{1/t}) \|.
\end{align*}
Lemma~\ref{sigmanormal} yields 
\[
 \sigma m' n' a_{1/t} = n'' a_r k
\]
for some $n''\in N$, $k\in K$ and 
\[
 r= \frac{\frac1t}{\left( \frac1t + \frac14|X|^2\right)^2 + |Z|^2}.
\]
Thus,
\[
\| v_\xi \varrho(ga_t) \| = \left(\frac{\frac1t}{\left( \frac1t + \frac14|X|^2\right)^2 + |Z|^2}\right)^{-q} s^{-q} \| v_\xi \varrho(\xi) \|.
\]
\end{proof}

The following proposition describes the amount of time a trajectory spends in a neighborhood of the cusp represented by $\xi$. 

\begin{prop}\label{sojourn}
Let $\xi\in\Xi$ and $g\in G$. Write $\delta\sceq\| v_\xi \varrho(g)\|$. If $g\in \xi NAM\sigma$, then $v_\xi\varrho(ga_t) \in B_{\delta}$ if and only if $t<1$. If $g=\xi n  a_s m u \in \xi NAMU$ with $u = \sigma (1, Z, X) \sigma$, then $v_\xi \varrho(ga_t) \in B_{\delta}$ if and only if
\[
 t\in \left( \frac{1}{\frac1{16}|X|^4 + |Z|^2 }, 1\right) \cup \left(1, \frac{1}{\frac1{16}|X|^4 + |Z|^2} \right).
\]
If $u=\id$, then $\big(\tfrac1{16}|X|^4 + |Z|^2\big)^{-1}$ is to be understood as $\infty$.
\end{prop}

\begin{proof}
The first part of the statement follows immediately from Lemma~\ref{value}. Suppose now that $g=\xi n m a_s u$ with $u=\sigma n' \sigma$ and $n'=(1,Z,X)$. By Lemma~\ref{value},
\begin{equation}\label{rform}
\| v_\xi \varrho(ga_r) \| = \left(\frac{\frac1r}{\left( \frac1r + \frac14|X|^2\right)^2 + |Z|^2}\right)^{-q} s^{-q} \| v_\xi \varrho(\xi) \|.
\end{equation}
Applying \eqref{rform} for $r=1$ and $r=t$, we see that
\[
 \| v_\xi \varrho(ga_t) \| < \| v_\xi \varrho(g) \|
\]
if and only if
\[
 \frac{1}{\left(1+\frac14|X|^2\right)^2 + |Z|^2} < \frac{\frac1t}{\left(\frac1t + \frac14|X|^2\right)^2 + |Z|^2},
\]
which is equivalent to
\[
 \left( 1 - \frac1t\right) \left( -\frac1t + \frac1{16}|X|^4 + |Z|^2 \right) < 0.
\]
This is the case if and only if
\[
 |Z|^2 + \frac{1}{16}|X|^4 <\frac1t < 1 \quad\text{or}\quad |Z|^2 + \frac{1}{16}|X|^4 > \frac1t > 1.
\]
\end{proof}

Suppose that $\|v_\xi\varrho(ga_t)\|=\|v_\xi\varrho(\gamma g a_t)\|$ for some $g\in G$, $\gamma\in\Gamma$ and all $t$ in a non-trivial interval (ie., an interval which contains at least two points). Then Lemma~\ref{value} yields that $g$ and $\gamma g$ have the same $A$-component in $\xi NAK$ and they are in the same $\xi$-Bruhat cell. If moreover, $g$ and $\gamma g$ are in the big $\xi$-Bruhat cell, then also the norms of their $U$-components are equal. The following lemma shows that far out in the cusp much more is true.

\begin{lemma}\label{stable}
Let $\xi\in\Xi$ and suppose that $g\in G$ and $\gamma \in \Gamma$ are such that 
\[
 \| v_\xi \varrho(g) \| = \| v_\xi \varrho(\gamma g) \| < \delta_\xi(s_1).
\]
Then $\gamma \in \xi NM \xi^{-1}$. In particular, if $g= \xi n  a m \sigma$ resp.\@ $g= \xi n  a m u$ with $n\in N$, $a\in A$, $m\in M$  and $u\in U$, then $\gamma g = \xi n' a m' \sigma$ resp.\@ $\gamma g = \xi n' a m'  u$ for some $n'\in N$, $m' \in M$.
\end{lemma}

\begin{proof}
By \cite[Lemma~2.2]{Dani} (see also Proposition~\ref{samenorm}), for each $s>0$ we have
\[
 L_\xi \xi A_s K =  \left\{ g\in G \left\vert\  v_\xi \varrho(g) \in B_{\delta_\xi(s)}\right.\right\}.
\]
Hence $g, \gamma g \in L_\xi \xi A_{s_1} K$. By \cite[Remark~1.3]{Dani} (with $\eta$ as in Proposition~\ref{funddomGR}), 
\[
 L_\xi \xi A_{s_1} K = ( \Gamma \cap L_\xi) \xi \eta A_{s_1} K.
\]
Hence there exist $\gamma_1,\gamma_2\in\Gamma\cap L_\xi$, $h_1,h_2 \in \eta A_{s_1} K$ such that
\[
 g=\gamma_1 \xi h_1, \quad \gamma g = \gamma_2 \xi h_2.
\]
Therefore
\[
 g \in \gamma_1 \xi \Omega(s_1,\eta) \cap \gamma^{-1}\gamma_2 \xi \Omega(s_1,\eta).
\]
Proposition~\ref{funddomGR}\eqref{funddomGRiv} yields $\gamma_1^{-1}\gamma^{-1}\gamma_2 \in \xi NM \xi^{-1}$. Thus, $\gamma\in \xi NM \xi^{-1}$.
\end{proof}

For the proof of the following proposition we recall that the supremum in the definition of $\xi$-height \eqref{xiheight} is realized if $\height_\xi(x) \geq s_1$.

\begin{prop}\label{ht_super}
Let $s > s_1$ and $x\in\homsp$. Suppose that there exists an interval $I$ in $\R$ such that $\height(xa_t) > s$ for all $t\in I$. Then there exists a unique cusp representative $\xi\in \Xi$ and a (non-unique) element $g\in G$ with $x=\Gamma g$ such that 
\[
 \height(xa_t) = \height_\xi(xa_t) = \left( \frac{\|v_\xi\varrho(ga_t)\|}{\|v_\xi\varrho(\xi)\|} \right)^{-\frac1q}
\]
for all $t\in I$. Moreover, if $1\in I$ and if there exists $t\in I$ with $t>1$ and $\height(xa_t) > \height(x)$, then $g=\xi n a_r m u$ for some $r>0$, $n\in N$, $m\in M$ and $u\in U$. The elements $a_r$ and $u$ do not depend on the choice of $g$. Finally, if $u=\sigma (1, Z, X) \sigma$, then 
\[
 |X| < 2 t^{-1/4} \quad\text{and}\quad |Z| < t^{-1/2}.
\]
\end{prop}

\begin{proof}
If $y\in \homsp$ and $\xi\in\Xi$ such that $\height_\xi(y) > s_1$, then there exists $h\in G$ such that $y=\Gamma h$ and 
\[
 \height_\xi(y) = \left( \frac{\|v_\xi\varrho(h)\|}{\|v_\xi\varrho(\xi)\|}\right)^{-\frac1q}.
\]
Since the function 
\[
\left\{
\begin{array}{ccl}
\R_{>0} & \to & \R
\\
r & \mapsto & \|v_\xi\varrho(ga_r)\| 
\end{array}
\right.
\]
is continuous, there exists an open neighborhood $J$ of $1$ in $\R_{>0}$ such that $\height_\xi(ya_r)>s_1$ for all $r\in J$. For $\xi\in \Xi$ let 
\[
J_\xi \sceq \big\{ t\in I \ \big\vert\ \height_\xi(xa_t) > s \big\}.
\]
These sets are pairwise disjoint, open in $I$ and cover $I$. Since $I$ is connected, there exists a unique $\xi \in \Xi$ with $I = J_\xi$. Thus
\[
 \height(xa_t) = \height_\xi(xa_t)
\]
for all $t\in I$. For each $t\in I$ pick an element $g_t\in G$ such that $x=\Gamma g_t$ and 
\[
 \height_\xi(xa_t) = \left( \frac{\|v_\xi\varrho(g_ta_t)\|}{\|v_\xi\varrho(\xi)\|}\right)^{-\frac1q}.
\]
Let $J_t$ be the set of $p\in I$ such that 
\[
\height_\xi(xa_{p}) = \left( \frac{\|v_\xi\varrho(g_ta_{p})\|}{\|v_\xi\varrho(\xi)\|}\right)^{-\frac1q}.
\]
Then $I$ is covered by the sets $J_t$, and these are open in $I$ by Proposition~\ref{samenorm}. If $J_t$ and $J_r$ overlap for some $t,r\in I$, $t\not=r$, then Lemma~\ref{stable} and \ref{value} imply that $J_t=J_r$. In turn, $J_t=I$ for each $t\in I$.

The remaining statements follow immediately from Proposition~\ref{sojourn} and Lemma~\ref{stable}.
\end{proof}

\section{Common cusp excursions of nearby points}

For $s>0$ we define
\begin{align*}
\homsp_{\leq s} & \sceq \homsp\setminus \homsp_{>s}.
\end{align*}
Further we let 
\[
 r_0 \sceq \alpha(\tilde a)
\]
and recall that~$r_0>1$ by our choice of~$\alpha$.

Each connected component of $\homsp$ of height above $s_1$ can essentially be identified with a Siegel set (cf.\@ Proposition~\ref{funddomGR}). For the proof of the main theorem, trajectories of points $x\in\homsp$ are only considered time-discretized by the map $T$. In the following lemma we construct a height level $s_2$ above which we can identify pieces of these discretized trajectories with trajectory segments in a Siegel set. More specifically, as soon as we know that two consecutive points of the discretized trajectory stay above height $s\geq s_2$, then the (continuous) trajectory segment of the corresponding geodesic also stays above height $s$ and, in particular, does not visit the compact set $\homsp_{\leq s_1}$. Then we construct a second height level $s_3>s_2$ such that any discretized trajectory entering $\homsp_{>s_3}$ can locally be identified with a continuous trajectory segment in the Siegel set. In Section~\ref{notfull} below this will be crucial to effectively determine the behavior of 
nearby starting trajectories. Of special importance for Section~\ref{sec_main} below is the item \eqref{descent} of the following lemma, which states that if we start to descend somewhere high in a cusp, then we actually descend up to below height $s_3$.

\begin{lemma}\label{s2s3}
There exist $s_3>s_2>s_1$ such that we have the following properties: 
\begin{enumerate}[{\rm (i)}]
\item\label{lowbound} If $x\in\homsp_{> s_2}$, then $\height(xa_t) > 2s_1$ for all $t\in [r_0^{-1},r_0]$.
\item\label{up} If $s\geq s_2$ and $x, Tx\in\homsp_{>s}$, then $\height(xa_t)>s$ for all $t\in [1,r_0]$.
\item\label{below} Let $s>s_3$. If $x\in\homsp_{\leq s_3}$ and $T^jx \in\homsp_{>s}$ for some $j\in \N$, then there exists $n\in\{0,\ldots, j-1\}$ such that $\height(T^nx) \leq s_3$ and $\height(xa_t) > s_2$ for all $t\in [r_0^n,r_0^j]$.
\item\label{above} Let $s>s_3$. If $x\in\homsp_{>s}$ and $T^jx\in\homsp_{\leq s_3}$ for some $j\in\N$, then there exists $n\in\{1,\ldots, j\}$ such that $\height(T^nx)\leq s_3$ and $\height(xa_t)> s_2$ for all $t\in [1,r_0^n]$.
\item\label{descent} Let $s>s_3$. If $x\in\homsp_{>s}$ and $Tx\in\homsp_{\leq s}$, then there exists $n\in\N$ such that $T^nx\in \homsp_{\leq s_3}$ and $T^kx\in\homsp_{\leq s}$ for all $k=1,\ldots, n$.
\end{enumerate}
\end{lemma}

\begin{proof}
We will choose $s_2 > s_1$ below.  Let $x\in\homsp_{>s_2}$. 
We wish to prove that $xa_t\in \homsp_{>2s_1}$ for all $t\in [r_0^{-1},r_0]$. 
Since $s_2>s_1$, there exist by Proposition~\ref{samenorm} a unique $\xi\in\Xi$ and an element $g\in G$ such that $x=\Gamma g$ and 
\[
 \height(x) = \height_\xi(x) = \left( \frac{\|v_\xi\varrho(g)\|}{\|v_\xi\varrho(\xi)\|}\right)^{-\frac1q}.
\]
Further, for all $t\in [r_0^{-1},r_0]$, we have
\[
 \height(xa_t) \geq \height_\xi(xa_t) \geq \left( \frac{\|v_\xi\varrho(ga_t)\|}{\|v_\xi\varrho(\xi)\|}\right)^{-\frac1q}. 
\]
However, now it is clear that if~$s_2$ is sufficiently big\footnote{A more careful
analysis using Lemma~\ref{value} reveals that~$s_2>2r_0^2s_1$ suffices.} or equivalently $\|v_\xi\varrho(g)\|$
is sufficiently small,  this will force~$\|v_\xi\varrho(ga_t)\|$ for~$t\in[r_0^{-1},r_0]$ sufficiently small to get the claim in~$(i)$.

For the proof of the remaining properties we will use the (quite natural) monotonicity properties of the functions appearing in Lemma~\ref{value}. 
So assume that
\[ 
 \height(x)=\height_\xi(x)= \left( \frac{\|v_\xi\varrho(g)\|}{\|v_\xi\varrho(\xi)\|}\right)^{-\frac1q}>s_1
\]
for~$x=\Gamma g$ (with the cusp representative $\xi$ and~$\pm v_\xi\varrho(g)$ uniquely determined by Proposition~\ref{samenorm}). 
If~$g=\xi na_sm\sigma$ is as in the first part of Lemma~\ref{value}, then the trajectory comes straight out of the cusp. 
Hence $\height(xa_t)=\frac{s}t$ is monotonically decreasing
until it reaches the value~$s_1$ (at which point Proposition~\ref{samenorm} will not apply any longer).
In the more general case, if $g=\xi na_sm\sigma(1,Z,X)\sigma$ is as in the second part of Lemma~\ref{value}, then the height of~$xa_t$
is given by the formula
\[
 \height(xa_t)=s\cdot\frac{ \frac1t }{ \left(\frac1t+\frac14|X|^2\right)^2 + |Z|^2 }
=\cdot\frac{s}{\frac1t+\frac14|X|^2+\left(|Z|^2+\frac{|X|^4}{16}\right)t },
\]
at least for all~$t$ for which the right hand side is~$\geq s_1$. If~$X=0$ and~$Z=0$ the right hand side equals~$st$ and
the orbit points straight into the cusp. However, in general the right hand side has a unique maximum,
is monotonically increasing left to the maximum and monotonically decreasing to the right of the maximum.

Property \eqref{lowbound} and these monotonicity properties imply \eqref{up}.

We choose~$s_3$ in the same way as~$s_2$ but with~$s_2$ replacing~$s_1$ in \eqref{lowbound}. Assume now~$s>s_3$,~$x\in \homsp_{\leq s_3}$
and~$T^j x\in\homsp_{>s}$ for some~$j\in\N$. We choose the maximal integer~$n<j$ with~$\height(T^nx)\leq s_3$.
By our choice of~$s_3$ we have~$\height(xa_t)>2 s_2$ for~$t\in[r_0^{n-1},r_0^{n+1}]$. Using 
the above monotonicity properties now implies \eqref{below}. Property \eqref{above} follows in the same way
using the first~$n\leq j$ with~$\height(T^nx)\leq s_3$. 

Property \eqref{descent} follows directly from the monotonicity properties.
\end{proof}

Given a point $x\in \homsp$ whose orbit stays near the cusp represented by $\xi$ for the next $S$ steps, Proposition~\ref{Vbound2} below provides non-trivial constraints on small perturbations of $x$ which do not destroy the qualitative behavior of the orbit for these next $S$ steps. The following lemma is needed for its proof.

\begin{lemma}\label{estsojournL2}
Let $D^U$ be a bounded subset of $U$. Let $\xi\in\Xi$ and $g = \xi n a_r m u \in G$ with $n\in N$, $a_r\in A$, $m\in M$ and $u=\sigma(1,Z,X)\sigma\in D^U$. Suppose that 
\[
 \left( \frac{\| v_\xi\varrho(ga_t)\|}{\|v_\xi\varrho(\xi)\|}\right)^{-\frac{1}{q}} > \lambda \left( \frac{\|v_\xi\varrho(g)\|}{\|v_\xi\varrho(\xi)\|}\right)^{-\frac{1}{q}}
\]
for some $t>1$ and $\lambda>0$. Then there exist $c_1, c_2 > 0$, only depending on $D^U$ and $\lambda$, such that 
\[
 |X| < c_1 t^{-1/4} \quad\text{and}\quad |Z| < c_2 t^{-1/2}.
\]
\end{lemma}

\begin{proof}
For $\lambda\geq 1$, the statement is already proven in Proposition~\ref{ht_super}. So suppose $1 > \lambda > 0$. Invoking Lemma~\ref{value} we find
\begin{equation}\label{maseq}
t \left[ \left(\frac1t + \frac14|X|^2\right)^2 + |Z|^2 \right] < \frac{1}{\lambda} \left[ \left(1 + \frac14|X|^2\right)^2 + |Z|^2 \right].
\end{equation}
Thus, 
\[
 t \left( \frac1t + \frac14|X|^2\right)^2 < \frac1{\lambda} \left(1+\frac14|X|^2\right)^2 + (\lambda^{-1} - t) |Z|^2.
\]
For $t > \lambda^{-1}$, it follows that
\[
 t \left( \frac1t+\frac14|X|^2\right)^2 < \lambda^{-1} \left( 1 + \frac14|X|^2\right)^2.
\]
Therefore,
\[
 |X|^2 < 4 \frac{ (t\lambda^{-1})^{\frac12} - 1}{ t^{\frac12} - \lambda^{-\frac12} } t^{-\frac12}.
\]
Hence, for $t>\lambda^{-1} + 1$,
we have
\[
 |X| < c_1 t^{-\frac14}
\]
for some constant $c_1 > 0$. Since $|X|$ is bounded, by possibly choosing a larger $c_1$, this estimate holds for all $t>1$.
To deduce the bound for $|Z|$ we note that \eqref{maseq} yields
\begin{align*}
 (t-\lambda^{-1}) |Z|^2 &< \lambda^{-1} \left(1+\frac14|X|^2\right)^2 - t\left(\frac1t + \frac14|X|^2\right)^2
\\
& < \lambda^{-1}(1+\frac14c_1^2)^2=c_3.
\end{align*}
Suppose that $t>\lambda^{-1} + 1$. 
Then
\[
 |Z|^2 < 
\frac{c_3}{t-\lambda^{-1}}=\frac{c_3}{1-(t\lambda)^{-1}}t^{-1}.
\]
The factor in front of $t^{-1}$ is bounded. Thus, 
\[
 |Z| < c_2 t^{-\frac12}
\]
for some constant $c_2 > 0$. As before, since $|Z|$ is bounded, this estimate holds for all $t>1$ after possibly choosing a larger $c_2$. This completes the proof.
\end{proof}

Let $d$ be a the left-$G$-invariant metric on $G$ induced from a left-invariant Riemannian metric that is induced by an inner product on $\mf g$. For $r>0$ let $B_r^G$ denote the open $d$-ball in $G$ centered at the identity of $G$ with radius $r$.  For $\kappa > 0$ let $D^U_\kappa$ denote the subset of $U$ consisting of the elements $u=\sigma (1,Z,X) \sigma$ with $|Z| < \kappa$ and $|X| < \kappa$, and let $D_\kappa^{NAM} \sceq B_\kappa^G \cap NAM$. Further let 
\begin{equation}\label{choicekappa}
 D_\kappa \sceq D_\kappa^U D_\kappa^{NAM}.
\end{equation}
Then $D_\kappa$ is open. We choose $\kappa>0$ such that for all $h\in D_\kappa$ we have
\begin{equation}\label{kappa}
 \| \varrho(h)\| , \| \varrho(h^{-1})\| \leq \left(\frac{s_1}{s_2}\right)^{-q}.
\end{equation}
We consider $\kappa$ to be fixed throughout and will shrink it if necessary (e.g.\@ in the paragraph before Lemma~\ref{maxexists}).

\begin{prop}\label{Vbound2}
There exist $c_3,c_4>0$  such that the following holds: Let $x\in\homsp$, $S\in\N$, $h\in D_\kappa$ be such that $\height(T^jx) > s_2$ and $\height(T^j(xh)) > s_2$ for $j=0,\ldots, S$, $\height(T^Sx) > \height(x)$ and $\height(T^S(xh)) > \height(xh)$. Suppose that $h=\sigma(1,Z,X)\sigma na_rm$. Then 
\[
 |X| \leq c_3 r_0^{-S/4} \quad\text{and}\quad |Z|\leq c_4 r_0^{-S/2}.
\]
\end{prop}

\begin{proof}
By Lemma~\ref{s2s3} we have $\height(xa_t)>s_2$ and $\height(xha_t)>s_2$ for all $t\in [1,r_0^S]$. Since $s_2 > s_1$,  Proposition~\ref{ht_super} shows that there exist a unique cusp representative $\xi\in\Xi$ and an element $g\in G$ such that $x=\Gamma g$ and 
\[
\height(xa_t) = \left( \frac{\|v_\xi\varrho(ga_t)\|}{\|v_\xi\varrho(\xi)\|}\right)^{-\frac1q}
\]
for all $t\in [1,r_0^S]$. Moreover, there exist a unique cusp representative $\xi_1\in\Xi$ and an element $g_1\in G$ such that $xh=\Gamma g_1 h$ and 
\[
\height(xha_t) = \left( \frac{\|v_{\xi_1}\varrho(g_1ha_t)\|}{\|v_{\xi_1}\varrho({\xi_1})\|}\right)^{-\frac1q} 
\]
for all $t\in [1,r_0^S]$. In the following we show that $\xi=\xi_1$ and that we can choose $g_1=g$.
We have
\[
\| v_\xi\varrho(gha_t) \|= \|v_\xi\varrho(ga_t a_{t^{-1}}ha_t) \| \leq \|v_\xi\varrho(ga_t)\|\cdot\|\varrho(a_{t^{-1}}ha_t) \|.
\]
Now, $a_{t^{-1}}ha_t \in D_\kappa$ for $t$ near $1$, say in the non-trivial interval $I$. By \eqref{kappa}, for $t\in I$ this yields
\[
 \|\varrho(a_{t^{-1}}ha_t) \| \leq \left( \frac{s_1}{s_2} \right)^{-q}.
\]
Thus, for $t\in I$, 
\begin{align*}
 \| v_{\xi}\varrho(gha_t)\| &\leq \| v_{\xi}\varrho(ga_t)\| \left(\frac{s_1}{s_2}\right)^{-q} 
 < \left(\frac{s_1}{s_2}\right)^{-q} s_2^{-q} \|v_{\xi}\varrho(\xi)\| 
\\
& = s_1^{-q} \| v_{\xi}\varrho(\xi)\|.
\end{align*}
Hence, for $t\in I$,
\[
\left( \frac{\| v_{\xi}\varrho(gha_t)\|}{\|v_{\xi}\varrho(\xi)\|}\right)^{-\frac1q} > s_1.
\]
The uniqueness of $\xi_1$ yields $\xi_1=\xi$. Moreover, we can choose $g_1=g$ for $t\in I$. As in the proof of Proposition~\ref{ht_super}, we see that we can choose $g_1=g$ for all $t\in [1,r_0^S]$.

Proposition~\ref{ht_super} shows that $g \in \xi NAMU$, say $g=\xi n_4 a_{r_1} m_1 u_1$ with $u_1=\sigma (1, Z_1, X_1)\sigma$, and that 
\begin{equation}\label{firstbound}
  |X_1| < 2 r_0^{-S/4} \quad\text{and}\quad |Z_1| < r_0^{-S/2}.
\end{equation}
Suppose that $h=u_2n_3a_{r_2}m_2$ and set $h_2\sceq n_3a_{r_2}m_2$. Then
\begin{align*}
\|v_\xi\varrho(gh)\| & = \|v_\xi\varrho(gu_2h_2)\| \leq \| v_\xi\varrho(gu_2)\| \|\varrho(h_2)\| 
 \leq \|v_\xi\varrho(gu_2)\| \left(\frac{s_1}{s_2}\right)^{-q}
\intertext{and}
\|v_\xi\varrho(gu_2a^S)\| & = \|v_\xi\varrho(gha^Sa^{-S}h_2^{-1}a^S)\| 
\\
& \leq \|v_\xi\varrho(gha^S)\| \|\varrho(a^{-S}h_2^{-1}a^S)\| \leq \| v_\xi\varrho(gha^S)\| \left(\frac{s_1}{s_2}\right)^{-q}.
\end{align*}
This yields
\begin{align*}
\left( \frac{\|v_\xi\varrho(gu_2a^S)\|}{\|v_\xi\varrho(\xi)\|}\right)^{-\frac1q} & \geq \frac{s_1}{s_2} \left(\frac{\|v_\xi\varrho(gha^S)\|}{\|v_\xi\varrho(\xi)\|}\right)^{-\frac1q} = \frac{s_1}{s_2} \height(xha^S) 
\\
& > \frac{s_1}{s_2}\height(xh) = \frac{s_1}{s_2}\left( \frac{\|v_\xi\varrho(gh)\|}{\|v_\xi\varrho(\xi)\|}\right)^{-\frac1q} 
\\
& \geq \left(\frac{s_1}{s_2}\right)^2 \left(\frac{\|v_\xi\varrho(gu_2)\|}{\|v_\xi\varrho(\xi)\|}\right)^{-\frac1q}.
\end{align*}
Let $u_2 = \sigma (1, Z_2, X_2)\sigma$. Then 
\[
 u_1u_2 = \sigma (1, Z_1+Z_2+\tfrac12 [X_1,X_2], X_1 + X_2) \sigma.
\]
From \eqref{firstbound} and $u_2\in D_\kappa^U$ it follows that
\[
 |X_1 + X_2| \leq |X_1| + |X_2| < 2 + \kappa.
\]
Moreover, using triangle inequality and \cite[Lemma~2.12, Proposition~3.3]{Pohl_isofunddom} we find
\[
 |Z_1 + Z_2 + \tfrac12[X_1,X_2]| \leq |Z_1| + |Z_2| + \tfrac12 |X_1| |X_2| < 1+ 2\kappa.
\]
Thus, $u_1u_2$ is contained in the bounded set $D^U_{2 + 2\kappa}$. Note that this set only depends on $\kappa$.
Then Lemma~\ref{estsojournL2} gives
\[
 |X_1+X_2| < c_1 r_0^{-S/4}\quad\text{and}\quad |Z_1+Z_2+\tfrac12[X_1,X_2]| < c_2r_0^{-S/2},
\]
where the constants $c_1,c_2$ only depend on $s_1,s_2$ and $\kappa$. It follows that 
\begin{align*}
 |X_2| &< c_1 r_0^{-S/4} + |X_1| < (c_1+2) r_0^{-S/4}
\intertext{and}
 |Z_2| &\leq |Z_1 + Z_2 + \tfrac12[X_1,X_2]| + |Z_1| + \tfrac12|X_1||X_2|
\\ 
& \leq c_2r_0^{-S/2} + r_0^{-S/2} + (c_1+2)r_0^{-S/2}.
\end{align*}
This completes the proof.
\end{proof}

\section{Estimate of metric entropy and proof of Theorem~A}\label{sec_main}

This section, in which we prove Theorem~A, can be understood independently from the previous ones if one is willing to accept the following facts previously shown: The height level $s_3$ is chosen such that the connected parts of $\homsp_{>s_3}$ (thus, cuspidal ends of uniform ``length'') can be identified with $(\Gamma\cap P)\backslash C$, where $C$ is the cylindrical set $C=\xi A_{s_3} NK$ at the cusp represented by $\xi$ of the considered end and $P$ is the corresponding minimal parabolic subgroup in $G$. In particular, this means that connected parts of geodesic trajectories in $\homsp_{>s_3}$ can be identified with any representing geodesic trajectories in $C$. As a consequence we know (see Lemma~\ref{s2s3}) that (discretized) geodesic trajectories in $\homsp_{>s_3}$ which start to move out of the cusp actually descend to below height level $s_3$, and geodesics in $\homsp$ which move from one of these cuspidal ends to another one necessarily have to pass through the compact part 
$\homsp_{\leq s_3}$. 
Moreover, if the trajectories of two nearby points $x, xh$ in $\homsp$ ($h\in G$) stay together near a cusp (meaning in the same connected component of $\homsp_{>s_3}$) for ``time'' $t$, then the unstable component of $h$ is restricted (up to a multiplicative constant) by $t^{-1/2}$ in the direction of the long root and by $t^{-1/4}$ in the direction of the short root (see Proposition~\ref{Vbound2}). 

Let $\mc M_1(\homsp)^T$ denote the set of $T$-invariant probability measures on $\homsp$. Let $\mu\in \mc M_1(\homsp)^T$ and suppose that $\mc P$ is a partition of $\homsp$ (consisting of measurable sets). We denote the static entropy of $\mc P$ with respect to $\mu$ by 
\begin{equation}\label{eq:static}
 H_\mu(\mc P) = - \sum_{P\in\mc P} \mu(P)\log\mu(P).
\end{equation}
For $n\in\N_0$ let
\[
 \mc P_0^n \sceq \bigvee_{j=0}^{n} T^{-j}\mc P = \big\{ P_{j_0} \cap T^{-1} P_{j_1} \cap \ldots \cap T^{-n}P_{j_n}\ \big\vert\ P_{j_i}\in\mc P\big\}.
\]
Then 
\[
 h_\mu(T,\mc P) = \inf_{n\in\N} \frac1n H_\mu\left( \mc P_0^{n-1} \right)
\]
is the dynamical entropy of $(T,\mc P)$ with respect to $\mu$. Finally,
\begin{align*}
 h_\mu(T) &= \sup\{ h_\mu(T,\mc P) \mid \text{$\mc P$ partition of $\homsp$, $H_\mu(\mc P) < \infty$}\}
\\
& = \sup\{ h_\mu(T,\mc P) \mid \text{$\mc P$  finite partition of $\homsp$}\}
\end{align*}
is the (metric) entropy of $T$ with respect to $\mu$.

In our set-up there exists a unique maximal entropy measure for $T$. We provide a reference for this statement and recall how to calculate its value in the following proposition. Set $p_1 \sceq \dim \mf g_1$, $p_2 \sceq \dim \mf g_2$ and recall that $\wt a = a_{r_0}$.

\begin{prop}\label{maxentropy}
The maximal entropy of $T$ is achieved by the Haar measure~$m$ on~$\homsp$ and is given by 
\[
 h_{m}(T) = \max\bigl\{ h_\mu(T) \mid  \mu\in\mc M_1(\homsp)^T\bigr\}= \left(\frac{p_1}{2}+p_2\right)\log r_0.
\]
Moreover, the Haar measure is the only~$T$-invariant probability measure that achieves this maximal entropy.
\end{prop}

\begin{proof}
The statement follows from a combination of the proposition in Section~9.3 in \cite{Margulis_Tomanov} and Lemma~9.5 and Proposition~9.6 in \cite{Margulis_Tomanov}. If $G$ is algebraic, a more accessible reference is \cite[Theorem~7.6]{Einsiedler_Lindenstrauss}. Note that 
\[
 -\log \det\left( \Ad_a\vert_{\mf g_{-1}\oplus\mf g_{-2}}\right) = \left(\frac{p_1}{2} + p_2\right)\log r_0. 
\]
\end{proof}

For $r>0$ we call
\begin{equation}\label{Bowengroup}
 B_L\sceq B_L(r)\sceq\bigcap_{j=0}^{L-1} \tilde a^j B_r^G \tilde a^{-j}
\end{equation}
a \textit{(forward) Bowen $L$-ball} in $G$ with \textit{(radius) parameter} $r$.
Further, any subset of $\homsp$ of the form
\begin{equation}\label{Bowenball}
 xB_L = xB_L(r)
\end{equation}
with $x\in\homsp$ is called a \textit{Bowen $L$-ball} in $\homsp$ with \textit{center} $x$ and \textit{(radius) parameter} $r$. 

Through the work of Brin--Katok \cite{Brin_Katok} it is well known that entropy is strongly related to the decay rate of
the measure of Bowen~$L$-balls. For the Haar measure this can be established quite directly and
in the following strong form (which will be used in many covering arguments below).

\begin{lemma}\label{massBowen}
Let $r>0$ be sufficiently small (depending only on~$G$) and $L\in\N$. Then 
\[
 r^{\dim G} e^{-h_m(T)L}\ll m(B_L(r)) \ll r^{\dim G} e^{-h_m(T)L},
\]
where the implied constants only depend on $G$ and $\tilde a$.
\end{lemma}

\begin{proof}
Recall from \eqref{choicekappa} the definition of $D_r$. We find $r_1,r_2>0$ (uniform for small $r$) such that
\[
 D_{r_1r} \subseteq B_r^G \subseteq D_{r_2r}.
\]
Then
\[
 D^{(L)}(r_1r)\sceq\bigcap_{j=0}^{L-1} \tilde a^j D_{r_1r} \tilde a^{-j} \subseteq B_L(r) = \bigcap_{j=0}^{L-1}\tilde a^j B_r^G \tilde a^{-j} \subseteq \bigcap_{j=0}^{L-1}\tilde a^j D_{r_2r} \tilde a^{-j}.
\]
One easily checks that
\[
 D^{(L)}(r) = \tilde a^{L-1} D_r^U \tilde a^{-(L-1)} D_r^{NAM}
\]
and
\[
 \tilde a^{L-1} D_r^U \tilde a^{-(L-1)} = \left\{ \sigma (1, Z, X) \sigma\ \left\vert\ |Z|< r r_0^{-(L-1)/2},\ |X|< r r_0^{-(L-1)/4} \right.\right\}.
\]
Let $du, dn, da$ and $dm$ be Haar measures on $U, N, A$ and $M$, respectively, and let $dg$ denote the Haar measure on $G$. With appropriate normalizations we have (\cite[Chapter~I, Proposition~5.21]{Helgason_gga}, and \cite[Chapter~I, Corollary~5.2]{Helgason_gga} for the change of order of integration)
\[
 \int_G f(g)dg = \int_{U\times N\times A\times M} f(unam) du dn da dm
\]
for all $f\in C_c(G)$. Further, we recall from \cite[Chapter~I, Theorem~1.14]{Helgason_gga} that if the support of $f\in C_c(G)$ is contained in the canonical coordinate neighborhood of $G$, then 
\begin{equation}\label{Jac}
 \int_G f(g) dg = \int_{\mf g} f(\exp W) \det\left( \frac{1-e^{-\ad W}}{\ad W}\right) dW,
\end{equation}
where $dW$ is the Euclidean measure on $\mf g$ which coincides with $(dg)_{\id}$. 

We now use the coordinates $(Z,X)\in \mf g_2\times \mf g_1$ for the Lie algebra $\mf u$ of $U$. Since $\mf u$ is two-step nilpotent, the Jacobian determinant in \eqref{Jac} (applied for $G=U$) equals $1$ for all $W\in\mf u$. With an appropriate global constant $c_U$, the Haar measure $du$ is then
\[
 m_U(f) \sceq \int_U f(u) du = c_U\int_{\mf u} f(\sigma(1,Z,X)\sigma)dZdX.
\]
Thus,
\[
 m_U\left( \tilde a^{L-1} D_r^U \tilde a^{-(L-1)} \right) = c_U r^{\dim U} e^{-h_m(T)(L-1)}.
\]
Hence
\[
 m\big(D^{(L)}(r_1r)\big) = c_U r_1^{\dim U} r^{\dim U} e^{-h_m(T)(L-1)} m_{NAM}\big( D_{r_1r}^{NAM} \big),
\]
where $m_{NAM} \sceq dn\otimes da\otimes dm$. We may assume that $D_r^{NAM}=B_r^{NAM}$. For sufficiently small $r>0$, the parameter space in $\mf n\times \mf a\times \mf m$ for the set $D_r^{NAM}$ is the spherical normal neighborhood $V_r = \{W\in\mf n\times \mf a\times \mf m \mid \|W\| < r\}$ (see \cite[Chapter~I, Proposition~9.4]{Helgason3}). On this neighborhood, the Jacobian determinant \eqref{Jac} (applied to $G=NAM$) is bounded from above and from below by some positive constants. Hence, \eqref{Jac} yields
\[
 r^{\dim NAM} \ll m_{NAM}\big( D_r^{NAM} \big) \ll r^{\dim NAM}.
\]
This completes the proof.
\end{proof}

We pick $\lambda >0 $ such that $r_0\lambda$ is an injectivity radius of $\homsp_{\leq s_3}$ and use it throughout as radius parameter for Bowen balls. Recall the set $D_\kappa$ and the choice of $\kappa$ from \eqref{choicekappa}-\eqref{kappa}. We may choose $\lambda$ so small such that $B_\lambda^G \subseteq D_\kappa$.

In Lemma~\ref{coverimp} below we will estimate how many Bowen $L$-balls are needed to cover $P \in \eta_{0}^{L-1}$, $P\subseteq \homsp_{\leq s_3}$, for certain partitions $\eta$ of $\homsp$.

\begin{lemma}\label{maxexists}
Let $s>s_3$. Then there exists $k_{\text{max}}\in\N$ such that whenever $x\in\homsp_{>s}$ satisfies $Tx,\ldots, T^kx\in \homsp_{\leq s}\cap \homsp_{>s_3}$, then $k\leq k_{\text{max}}$.
\end{lemma}

\begin{proof} 
Let $x\in\homsp_{>s}$ be as in the statement of the lemma. Lemma~\ref{s2s3} and Proposition~\ref{ht_super} show that there is a unique $\xi\in\Xi$ and some $g\in G$ such that $\Gamma g = x$ and 
\[
 \height(xa_t) = \left( \frac{\|v_\xi\varrho(ga_t)\|}{\|v_\xi\varrho(\xi)\|}\right)^{-\frac1q}
\]
for $t\in [1,r_0^k]$. We suppose first that $g=\xi n a_r  m\sigma$ for some $n\in N$, $r>0$ and $m\in M$. Then 
$\height(xa_t) = \frac{r}{t}$
for $t\in [1,r_0^k]$. Therefore $ s \geq \height(Tx) = \frac{r}{r_0}$.
This and  $\height(T^kx) = \frac{r}{r_0^k} > s_3$
yield
\[
 k < \frac{\log \frac{r}{s_3} }{\log r_0} \leq \frac{ \log \frac{sr_0}{s_3} }{\log r_0}.
\]

Now we suppose that $g= \xi n a_r  m \sigma (1,Z,X)\sigma$ for some $n, (1,Z,X)\in N$, $r>0$ and $m\in M$. For $t\in [1,r_0^k]$ we have
\[
 \height(xa_t) = r \cdot \frac{t^{-1}}{ \left( t^{-1} + \frac14 |X|^2 \right)^2 + |Z|^2 }.
\]
Then $\height(xa_t) > s_3$ is equivalent to 
\begin{equation}\label{rootcond}
 0 > \left( t^{-1} - \lambda_-\right) \left(t^{-1} - \lambda_+\right)
\end{equation}
where
\[
 \lambda_{\pm} = -\frac12\left(\frac12|X|^2 - \frac{r}{s_3}\right) \pm \sqrt{ \frac14\left(\frac12|X|^2 - \frac{r}{s_3}\right)^2 - \left( \frac{1}{16}|X|^4 + |Z|^2\right) }.
\]
Since $\height(x)>s_3$, \eqref{rootcond} is satisfied at least for $t=1$. Therefore, the roots $\lambda_{\pm}$ are real and 
\[
 \lambda_+ > 1 > \lambda_-.
\]
From $\lambda_+ > 1$ it follows that 
\[
 \frac12\left( \frac{r}{s_3} - \frac12|X|^2\right) > 0.
\]
In turn, $\lambda_->0$. Now $\height(T^kx)>s_3$ implies
\[
 r_0^k < \lambda_-^{-1} = \frac{\frac12\left(\frac{r}{s_3} - \frac12|X|^2\right) + \sqrt{ \frac14\left(\frac{r}{s_3}-\frac12|X|^2\right)^2 - \left(\frac1{16}|X|^4 + |Z|^2\right)}}{\frac1{16}|X|^4 + |Z|^2}.
\]
From $s\geq \height(Tx)$ it follows that 
\[
 r\leq r_0s \left[ \left(r_0^{-1} + \frac14|X|^2\right)^2 + |Z|^2\right].
\]
Therefore
\begin{align*}
\lambda_-^{-1} & \leq \frac{r_0s}{s_3} \cdot \frac{\left(r_0^{-1}+\frac14|X|^2\right)^2 + |Z|^2}{\frac1{16}|X|^4 + |Z|^2}
\\
& = \frac{s}{s_3} \cdot \frac{r_0^{-1}+ \frac12|X|^2}{\frac1{16}|X|^4 + |Z|^2} + \frac{r_0s}{s_3}.
\end{align*}
From $\height(x) > \height(Tx)$, a straightforward deduction yields
\[
 \frac{1}{16}|X|^4 + |Z|^2 > r_0^{-1}.
\]
Hence,
\[
 \frac{r_0^{-1}+ \frac12|X|^2}{\frac1{16}|X|^4 + |Z|^2}
\]
is bounded from above (independent of $x$), and so is $\lambda_-^{-1}$. This completes the proof.
\end{proof}

In the following, for $s'>s>s_3$ we define various numbers which vary with $s$ and $s'$. 

Let $\ell$ denote the maximal number of $T$-steps between $\homsp_{>s}$ and $\homsp_{\leq s_3}$, that is,
\begin{equation}\label{defell}
 \ell \sceq \max\{ k\in\N \mid \exists\, x\in \homsp_{>s}\colon Tx,\ldots, T^kx\in\homsp_{>s_3}\cap \homsp_{\leq s},\ T^{k+1}x \in \homsp_{\leq s_3}\}.
\end{equation}
We note that this maximum exists by Lemma~\ref{maxexists}. It equals the maximal number of $T$-steps between $\homsp_{\leq s_3}$ and  $\homsp_{>s}$ in the sense that 
\[
 \ell \sceq \max\{ k\in\N \mid \exists\, x\in \homsp_{\leq s_3}\colon Tx,\ldots, T^kx\in\homsp_{>s_3}\cap \homsp_{\leq s},\ T^{k+1}x \in \homsp_{>s}\},
\]
which follows since all sets of the form $\homsp_{\leq t}$ or $\homsp_{>t}$ are invariant under $\sigma$ and since $\sigma \tilde a \sigma=\tilde a^{-1}$. 
We note that the maximal amount of time a trajectory can spend continuously within $\homsp_{>s_3}\cap\homsp_{\leq s}$ is then bounded by $2\ell+5$ (corresponding to a trajectory that reaches about height $s$ and then returns to $\homsp_{\leq s_3}$), i.e.\ that
\[
 \max\{k\in\N \mid \exists\, x\colon x, Tx,\ldots, T^kx\in\homsp_{>s_3}\cap \homsp_{\leq s}\}\leq 2\ell+5 
\]
We define $\ell'$ in the same way using $s'$ in place of $s$. 

Let $s > s_3$ and $L\in\N$. Let $\eta$ be a finite partition of $\homsp$ of the form
\[
\eta = \{ \homsp_{>s}, \homsp_{>s_3}\cap \homsp_{\leq s}, P_1,\ldots, P_r\}
\]
with $P_i\subseteq \homsp_{\leq s_3}$ for $i=1,\ldots,r$.
For any $P \in \eta_{0}^{L-1}$ we define
\begin{equation}\label{defVP}
 V_P \sceq \left\{ j \in \{0,\ldots, L-1\} \left\vert\  T^jP \subseteq \homsp_{>s}\right.\right\}.
\end{equation}
For brevity we use the notation 
\[
 [m,n) \sceq \{m,m+1,\ldots, n-1\}
\]
for an interval of integer points with endpoints $m\leq n\in\N$.

An interval $k+[0,K)\subseteq [0,L)$ of a trajectory of a set $P\in\eta_0^{L-1}$ is said to be an \emph{excursion into $\homsp_{>s}$} (of length $K$) if 
\begin{align*}
& T^{k-1}P\subseteq X_{\leq s},\quad T^kP,\ldots, T^{k+K-1}P\subseteq \homsp_{>s},
\\
& \text{and either $T^{k+K}P\subseteq \homsp_{\leq s}$ or $k+K = L$.}
\end{align*}
Clearly, $V_P$ is a disjoint union of intervals which are excursions into $\homsp_{>s}$. 

For the statement of the following lemma we remark that $\homsp_{\leq s}$ is compact by \cite[p.~27]{Dani}. Further we recall that $\lambda$ is the parameter used in the definition of Bowen balls, and that $B_\lambda^G\subseteq D_\kappa$.

\begin{lemma}\label{coverimp}
Let $s'>s>s_3$ and define $\ell,\ell'$ as above. Let $\lambda'\in (0,\lambda]$ be such that $r_0\lambda'$ is an injectivity radius of $\homsp_{\leq s'}$. Suppose that $\eta = \{ \homsp_{>s}, \homsp_{>s_3}\cap \homsp_{\leq s}, P_1,\ldots, P_r\}$ is a finite partition of $\homsp$ such that $\diam T^j(P_i)\leq\lambda'$
for each $j=0,\ldots, 2\ell'+5$ and $i=1,\ldots,r$.
Then for each $L\in \N$ and $P \in \eta_{0}^{L-1}$ with $P\subseteq \homsp_{\leq s_3}$ the set $P$ can be covered by
\[
 c^m e^{h_m(T)\ell m}e^{\frac12 h_m(T) |V_P|}
\]
Bowen $L$-balls. Here the constant $c$ only depends on $G, r_0, s_1,s_2$ and $\lambda$. The constant $m$ (not to be confused with the Haar measure $m$) is the number of excursions of $P$ into $\homsp_{>s'}$.
\end{lemma}

We note that while the partition $\eta$ is (in a strong way) adapted to the heights $s,s'$,  our definition of Bowen $L$-ball does not depend on $s,s'$.

\begin{proof}
Let $P \in \eta_{0}^{L-1}$ with $P\subseteq \homsp_{\leq s_3}$.  We decompose $V_P$ into a disjoint union
of excursions into~$\homsp_{>s}$. We denote those intervals that contain excursions
into~$\homsp_{>s'}$ by~$V_j=[k_j,k_j+K_j)$ and their union by 
\begin{equation}\label{decompV}
 V = \bigcup_{j=1}^{m} V_j 
 = [k_1,k_1+K_1) \cup \ldots \cup [k_{m}, k_{m} + K_{m})\subset V_P.
\end{equation}
We may suppose that $k_1<k_2 < \cdots < k_{m}$. 
We note that each of those excursion $V_j$ is contained in an excursion 
$\tilde V_j=[n_j,n_j+h_j)\subseteq [k_j-\ell,k_j+K_i+\ell)$ into $\homsp_{>s_3}$. 
We define $\tilde V=\bigcup_{j=1}^{m}\tilde V_{j}$ (which is disjoint union). 
Analogously, we decompose $W \sceq [0,L)\smallsetminus \tilde V$ into a disjoint union 
\[
 W = \bigcup_{j=1}^{m+1} W_j
\]
where each $W_j$ is a maximal subset of $W$ of the form $[l_j, l_j+L_j)$ with $0=l_1<l_2<\cdots < l_{m+1}$. The set $W_{m+1}$ might be empty.
It follows that~$[0,L)$ is the disjoint union of~$W_1,\tilde V_1,\ldots,\tilde V_m,W_{m+1}$ in that order.

\begin{figure}[h]
\begin{center}
\includegraphics*{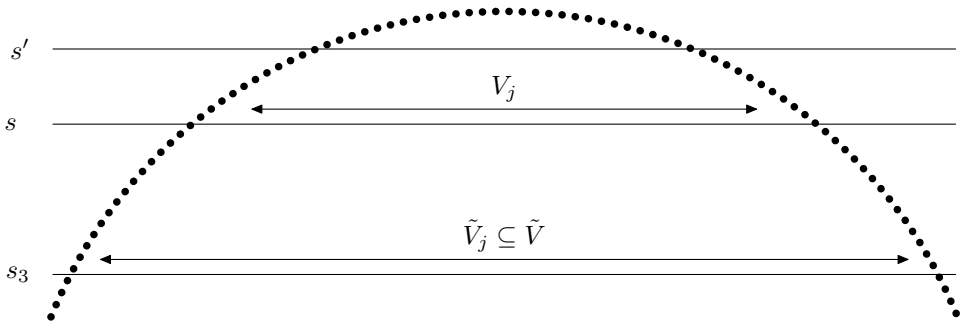} 
\end{center}
\caption{An excursion~$V_j$ into~$\homsp_{>s}$ containing an excursion into~$\homsp_{>s'}$
and contained in the excursion $\tilde V_j$ into~$\homsp_{>s_3}$.}
\end{figure}

For $R$ running iteratively through $|W_1|$, $|W_1\cup \tilde V_1|$, $\ldots$, $|W_1\cup \tilde V_1 \cup \ldots \cup W_m \cup \tilde V_m|$, $|W_1\cup \tilde V_1 \cup \ldots \cup W_m\cup \tilde V_m \cup W_{m+1}|$, we now show that the number of $R$-boxes with center in $P$ needed to cover $P$ is bounded by 
\[
 c^{j-1} e^{\frac12 h_m(T)|\tilde V_1\cup\ldots \cup \tilde V_{j-1}|} 
\]
for $R \in \{ |W_1 \cup \tilde V_1 \ldots \cup W_j|, |W_1 \cup \tilde V_1 \ldots \cup W_{j-1} \cup \tilde V_{j-1}| \}$, $j=1,\ldots, m+1$, where $c$ is a constant depending only on $G, \tilde a,s_1,s_2$ and $\lambda$, but not on $s$ or $s'$. We determine $c$ in \eqref{finalc} below.

\textit{The step corresponding to adding $W_j$:} 
Note first that even though an interval $W_j=[l_j, l_j+L_j)$ may contain one or more excursions  into $\homsp_{>s}$, it does not contain an excursion into $\homsp_{>s'}$. Each of the excursions into $\homsp_{>s}$ is of length at most $2\ell'+5$.

Suppose now that we have already found (at most)
\begin{equation}\label{njeq}
 c^{j-1} e^{\frac12 h_m(T) |\tilde V_1\cup\cdots\cup\tilde V_{j-1}|}
\end{equation}
Bowen $l_j$-balls with center in $P$ whose union contains the given element $P\in\eta_0^{L-1}$.
If $j=1$, and hence $l_j=0$, we may define a Bowen $0$-ball as a Bowen $1$-ball and the claim on the  covering number is trivial. We claim that $T^{l_j}(P)$, $\ldots$, $T^{l_j+L_j-1}(P)$ being contained in $\homsp_{\leq s'}$ implies that $T^{l_j}P$ is so small such that if $P$ is covered by the Bowen $l_j{\color{green}}$-balls with centers $x_1,\ldots, x_r \in P$, then $P$ is already covered by the smaller Bowen $l_j+L_j$-balls with these centers.

The choice of $\lambda'$, $\ell'$ and $\eta$ implies that any element $Q\in\eta$ with $Q\subseteq \homsp_{\leq s_3}$ is contained in 
\begin{equation}\label{smallball}
 w \bigcap_{i=0}^k \tilde a^i B_{\lambda'}^G \tilde a^{-i}
\end{equation}
for $k\leq 2\ell'+5$ and for any $w\in Q$ as long as $Q, T(Q), \ldots, T^k(Q) \subseteq \homsp_{\leq s'}$. A simple induction shows that $Q\in\eta_0^{L'-1}$, $Q\subseteq \homsp_{\leq s_3}$, belongs to the ``small'' Bowen $L'$-ball \eqref{smallball} (defined with the radius~$\lambda'$)
with $k=L'-1$ if it is known that $Q,T(Q),\ldots,T^{L'-1}(Q)\subseteq \homsp_{\leq s'}$. This applies to the partition element $Q$ containing $T^{l_j}(P)$ and $L'=L_j$, and shows that $T^{l_j}(P)$ is contained in 
\begin{equation}\label{actualcover}
 w \bigcap_{i=0}^{L_j-1} \tilde a^i B_{\lambda'}^G \tilde a^{-i}
\end{equation}
for any $w\in T^{l_j}(P)$. Let $xB_{l_j}$ be one of the Bowen balls
used in the cover associated to \eqref{njeq}. Set $w \sceq x\tilde a^{l_j}$. From $x\tilde a^{l_j} \in \homsp_{\leq s_3}$ and the choice of $\lambda'\leq\lambda$ it follows that
\begin{align*}
P  \cap xB_{l_j}
 &\subseteq \Big((x \tilde a^{l_j})  \bigcap_{i=0}^{L_j-1} \tilde a^i B_{\lambda'}^G \tilde a^{-i} \cap    (x\tilde a^{l_j}) \tilde a^{-l_j}B_{l_j}\tilde a^{l_j}\Big) \tilde a^{-l_j}
\\
&\subseteq (x \tilde a^{l_j})\Big(  \bigcap_{i=0}^{L_j-1} \tilde a^i B_{\lambda'}^G \tilde a^{-i} \cap   \tilde a^{-l_j}B_{l_j}\tilde a^{l_j}\Big) \tilde a^{-l_j}\subseteq xB_{l_j+L_j}.
\end{align*}
Note that $n_j=l_j+L_j$, which is the starting point of $\tilde V_j$ for $j\leq m$. Thus, $P$ can be covered with 
the same number of Bowen $n_j$-balls with center in $P$ as with Bowen~$l_j$-balls, so that \eqref{njeq} is still an upper bound
for the necessary number.

\textit{The step corresponding to adding $\tilde V_{j}$:}
Suppose now that we have already found \eqref{njeq}-many Bowen $n_j$-balls with center in $P$ whose union covers $P$. 
Let $xB_{n_j}$ be any Bowen $n_j$-ball used in this covering. Define
\[
 E \sceq \left\{ y \in xB_{n_j} \left\vert\ T^{n_j-1}y\in \homsp_{\leq s_3}, T^{n_j}y,\ldots, T^{n_j+h_j-1}y \in \homsp_{>s_3}\right.\right\}.
\]
Then $P\cap xB_{n_j} \subseteq E$. If $y\in E$, then there exists $g\in B_{n_j}$ with $y=xg$.
This implies $h\sceq \tilde a^{-n_j+1}g\tilde a^{n_j-1}=\sigma(1,Z,X)\sigma g'\in B_\lambda^G$
for some bounded (with the bound only depending on $G$ and $\lambda$)~$g'\in NAM$ and~$\sigma (1,Z,X)\sigma\in U$.
Set $x'\sceq T^{n_j-1}x$ and $y'\sceq T^{n_j-1}y$ so that $y'=x'h$.
Moreover, we have
\begin{align*}
& Tx',\ldots, T^{h_j}x'\in \homsp_{>s_3},\quad x'\in\homsp_{\leq s_3}\cap \homsp_{>s_2},
\\
& Ty',\ldots, T^{h_j}y'\in \homsp_{>s_3},\quad y'\in\homsp_{\leq s_3}\cap \homsp_{>s_2}.
\end{align*}
Applying Proposition~\ref{Vbound2} to $x',y'$ shows that 
\begin{equation}\label{main-bound-XZ}
 |Z| \leq c_4 r_0^{-h_j/2}\quad\text{and}\quad |X| \leq c_3 r_0^{-h_j/4}.
\end{equation}

We claim that this bound implies that we can cover~$P\cap xB_{n_j}$ with~$\leq c r_0^{(\frac{p_1}4+\frac{p_2}2)h_j}$
many Bowen~$l_{j+1}$-boxes with centers in~$P$ (recall that~$l_{j+1}=n_j+h_j$).
Here~$c$ is some constant that does not depend on~$s,s'$. 

To see that notice first that~\eqref{main-bound-XZ} implies that~$\tilde a^{-\lceil h_j/2\rceil}h \tilde a^{\lceil h_j/2\rceil}$
is still bounded uniformly in $y$ and $j$. More precisely we may accomodate the constants~$c_3,c_4$ by finding some
absolute integer~$b>0$ (depending on~$r_0,c_3,c_4$ and~$\lambda$) such that~\eqref{main-bound-XZ}
implies that~$\tilde a^{-\lceil h_j/2\rceil+b}\sigma (1,Z,X) \sigma \tilde a^{\lceil h_j/2\rceil-b}\in B_{\lambda}^G$.
Note that we also have
$$
 g',\tilde a^{-\lceil h_j/2\rceil+b} g'\tilde a^{\lceil h_j/2\rceil-b}\in B_{c'\lambda}^G
$$ 
for some constant~$c'>0$
that only depends on the choice of the Riemannian metric on~$G$ and $b$. Together we
see that $g\in B_{n_j}$ actually belongs to the Bowen~$(n_j+\lceil h_j/2\rceil-b)$-ball $B'$
defined by the radius~$(1+c')\lambda$. Define~$B''$ to be the Bowen $l_{j+1}$-ball
defined by the radius~$\lambda/2$. Let now~$y_1,\ldots,y_q\in P\cap x B'$
be a maximal collection of points for which the sets~$y_1B'',\ldots,y_q B''$ are pairwise disjoint.
By definition we have~$y_1B''\cup\cdots\cup y_q B''\subseteq xB'B''$. Also note that~$B'B''\subseteq B^{(3)}$,
where~$B^{(3)}$ is the Bowen~$(n_j+\lceil h_j/2\rceil-b)$-ball defined by the radius~$(3/2+c')\lambda$.
With the lower and upper bounds for the Haar measure of a Bowen~$n$-ball of the form~$ce^{-h_m(T)n}$ from Lemma~\ref{massBowen} for two values of $c$ (which depend on the radius used) we obtain
\begin{equation}\label{finalc}
 q\leq \frac{m(B^{(3)})}{m(B'')}\leq c''e^{-h_m(T)(n_j+\lceil h_j/2\rceil-b)+h_m(T)l_{j+1}}.
\end{equation}
Finally note that by maximality of the collection~$y_1,\ldots,y_q$
it follows that~$P\cap xB$ is covered by~$y_1B_{l_{j+1}},\ldots,y_qB_{l_{j+1}}$,
which gives the claim (by recalling that~$l_{j+1}=n_j+h_j$).

Finally note that the bound \eqref{njeq} for $j=m$ (if $W_{m+1}=\emptyset$) or 
for $j=m+1$ gives the conclusion of the proposition since $|\tilde V_j| \le |V_j|+2\ell$. 
\end{proof}

\begin{prop}\label{commonpartition}
For all $s > s_3$ there exists a finite partition $\eta = \{\homsp_{>s}, \homsp_{>s_3}\cap \homsp_{\leq s}, P_1,\ldots, P_r\}$ of $\homsp$ such that for each $T$-invariant probability measure $\mu$ on $\homsp$ we have
\[
 h_\mu(T) \leq  h_\mu(T,\eta) + \tfrac1s + \tfrac12 h_m(T) \big( 1-\mu(\homsp_{\leq s})\big).
\]
\end{prop}

\begin{proof}
Using the ergodic decomposition of $T$-invariant measures we may restrict to ergodic $T$-invariant measures. Also note that every $T$-orbit visits $\homsp_{\leq s_3}$ (Lemma~\ref{s2s3}\eqref{descent}, which also holds for $T^{-1}$ in place of $T$), so that we have 
\begin{equation*}
\delta_1\sceq\mu(\homsp_{\leq s_3})> 0.
\end{equation*}
  Let $s'>s$. Define $\ell$ as in \eqref{defell} and let
\[
 \ell''\sceq \min\left\{ k\in\N\left\vert\ \exists\, x\in \homsp_{\leq s_3}\colon T^{k+1}x\in \homsp_{>s'} \right.\right\}.
\]
Let $\eta = \{\homsp_{>s}, \homsp_{>s_3} \cap \homsp_{\leq s}, P_1,\ldots, P_r\}$ be as in Lemma~\ref{coverimp}. 

We will show the proposition using \cite{Brin_Katok}, more precisely in the form of Lemma~B.2 in \cite{ELMV}: There it is shown that for any small $\delta > 0$, the entropy of $\mu$ is the limit as $\lambda\to 0$
of
\[
 \liminf_{L\to\infty}\frac{\log N_\lambda(\delta,L)}L,
\]
where $N_\lambda(\delta,L)$ is the minimal number of Bowen $L$-balls, with $\lambda$ being the parameter used in the definition, that are needed to cover some set of $\mu$-measure $\delta$. However, there is one important difference between our definition of Bowen $L$-balls and that of \cite{Brin_Katok}. In the latter, one takes the intersections of pre-images of $\lambda$-balls within $\homsp$. In our definition of Bowen $L$-ball we took the intersection in the group, which results in general in a smaller set (namely in those cases where the orbit ventures near the cusp). As we are seeking an upper bound of entropy, we may use \cite{Brin_Katok} also together with our definition of Bowen $L$-balls. This has the advantage that we do not have to take the limit as $\lambda\to 0$. Within the group a bounded number of translates\footnote{Let~$\lambda<\lambda'$.
If~$g_1B_L(\lambda/2),\ldots,g_bB_L(\lambda/2)$ is a maximal collection of disjoint left translates for some~$g_1,\ldots,g_b\in B_L(\lambda')$,
then~$b\leq\frac{m\big(B_L(\lambda'+\lambda/2)\big)}{m\big(B_L(\lambda/2)\big)}$ is bounded 
independently (see Lemma~\ref{massBowen}) of~$L$ and~$B_L(\lambda')\subseteq g_1B_L(\lambda)\cup\cdots\cup g_bB_L(\lambda)$.} 
of Bowen $L$-balls defined by $\lambda>0$ can be used to cover a Bowen $L$-ball defined by $\lambda'>0$, and as $L\to\infty$ this difference becomes unimportant.

By ergodicity $\bigcup_{j=0}^\infty T^{-j}\homsp_{\leq s_3}$ has full measure. Hence there exists $M$ with 
\[
 \mu\bigg(\bigcup_{j=0}^{M-1} T^{-j}\homsp_{\leq s_3} \bigg)>1-\frac{\delta_1}2.
\]
The intersection of the preimage of this set under $T^{L'}$ with $\homsp_{\leq s_3}$ has measure at least $\delta_1/2$. It follows that there are infinitely many $L$ (of the form $L'+j$ for some $j\in [0,M)$) for which 
\[
 \mu\big(\homsp_{\leq s_3}\cap T^{-L}\homsp_{\leq s_3}\big)>\frac{\delta_1}{2M}.
\]
We now proceed making $Y_L\sceq\homsp_{\leq s_3}\cap T^{-L}\homsp_{\leq s_3}$ smaller, taking care that the resulting sets have measures that do not approach zero, and obtaining more information on the smaller sets. 

Since $\mu$ is ergodic, the value $h_\mu(T,\eta)$ has the following interpretation: for every $\eps>0$ and every sufficiently large $L$ there exists a set $Z_{L,\eps}$ such that its measure is bigger than $1-\eps$ and $Z_{L,\eps}$ can be covered with $e^{(h_\mu(T,\eta)+\eps)L}$ elements of $\eta_0^{L-1}$. We choose $\eps=\min(\frac1{3s}, \frac{\delta_1}{4M})$, and take the intersection $Y_L'=Y_L\cap Z_{L,\eps}$. We now know that $\mu(Y_L')>\frac{\delta_1}{4M}$ and that $Y_L'$ can be covered with $e^{(h_\mu(T,\eta)+\frac1{3s})L}$ elements of $\eta_0^{L-1}$. 

Finally, we may make $Y_L'$ again a bit smaller to ensure that the ergodic averages for the characteristic function $\homsp_{>s}$ are correct up to an error of $(3sh_m(T))^{-1}$ and for sufficiently large $L$. More precisely there exists a subset $Y_L''\subset Y_L'$ (obtained by intersecting $Y_L'$ with a set of near full measure) with $\mu(Y_L'')>\delta=\frac{\delta_1}{5M}$ and some $L_0$ such that for all $L>L_0$ and all $x\in Y''_L$ we have
\[
 \left|\frac1L\sum_{i=0}^{L-1}\chi_{\homsp_{>s}}(T^ix)-\mu(\homsp_{>s})\right|<\frac1{3sh_m(T)}.
\]
In the following we assume~$L>L_0$.
Now apply Lemma \ref{coverimp} to each of the partition elements of $\eta_0^{L-1}$ obtained earlier. Notice that for each of the partition elements we have $m\leq \frac{L}{\ell''}$. Finally, notice that the above ergodic sum is constant on each $P\in\eta_0^{L-1}$. For those $P$ that intersect $Y_L''$ we then have 
\[
|V_P|<(\mu(\homsp_{>s})+\frac1{3sh_m(T)})L.
\]
Thus, Lemma~\ref{coverimp} together with the above covering estimate on $Y_L'$ implies that $Y_L''$ can be covered with $N_L$ Bowen $L$-balls, where 
\[
 N_L \leq e^{(h_\mu(T,\eta)+\frac1{3s})L} (ce^{h_m(T)\ell})^{\frac{L}{\ell''}} e^{\frac12h_m(T)\mu(\homsp_{>s})L+\frac1{6s}L}.
\]
Choose $s'$ so big such that $(\log c +h_m(T)\ell)/\ell''<\frac1{6s}$. This implies the proposition.
\end{proof}

We now restate and prove Theorem~A from the introduction.

\begin{thm}\label{entropycusp} 
Let $(\mu_j)$ be a sequence of $T$-invariant probability measures on $\homsp$ which converges to the measure $\nu$. Then 

\[
 \nu(\homsp) h_{\frac{\nu}{\nu(\homsp)}}(T) + \tfrac12 h_m(T)\big(1-\nu(\homsp)\big) \geq \limsup_{j\to\infty} h_{\mu_j}(T),
\]
where it does not matter how we interpret $h_{\frac{\nu}{\nu(\homsp)}}(T)$ if $\nu(\homsp)=0$.
\end{thm}

\begin{proof}
Pick $s>s_3$ such that $\nu(\partial \homsp_{\leq s}) = 0$ (this holds for all but countably many $s$). Let $\eta = \{ \homsp_{>s}, \homsp_{>s_3}\cap\homsp_{\leq s}, P_1,\ldots, P_r\}$ be a partition of $\homsp$ as in Proposition~\ref{commonpartition} such that $\nu(\partial P_j)=0$ for $j=1,\ldots, r$. Let $\eps > 0$. Suppose now that $\nu(\homsp)>0$. By definition of entropy we may fix $m\in\N$ such that 
\[
 h_{\frac{\nu}{\nu(\homsp)}}(T) + \eps > \frac1m H_{\frac{\nu}{\nu(\homsp)}}(\eta_0^{m-1})
\]
and 
\[
\frac{2e^{-1}}{m} < \frac{\eps}{2}\quad\text{and}\quad -\frac1m\log \nu(\homsp) < \eps.
\]
Then using~\eqref{eq:static} we get
\[
 \nu(\homsp) h_{\frac{\nu}{\nu(\homsp)}}(T) + 2\eps > -\frac1m\sum_{P\in\eta_0^{m-1}} \nu(P) \log\nu(P).
\]
Note that this holds trivially if $\nu(\homsp)=0$.

Let
\[
 Q\sceq \bigcap_{k=0}^{m-1} T^{-k} \homsp_{>s}.
\]
Since
\[
\sum_{P\in\eta_0^{m-1}\setminus\{Q\}} \mu_j(P) \log\mu_j(P) \stackrel{j\to\infty}{\longrightarrow} \sum_{P\in\eta_0^{m-1}\setminus\{Q\}} \nu(P)\log\nu(P),
\]
we find $j_0\in\N$ such that for all $j\geq j_0$ we have
\begin{align*}
&\left|-\frac1m \sum_{P\in\eta_0^{m-1}} \nu(P)\log\nu(P)  -\frac1m H_{\mu_j}(\eta_0^{m-1})\right|
\\
& \quad\leq \frac1m \left| \sum_{P\in\eta_0^{m-1}\setminus\{Q\}} \big( \mu_j(P)\log\mu_j(P) - \nu(P)\log\nu(P)\big)\right| 
\\
& \hphantom{\sum_{P\in\eta_0^{m-1}\setminus\{Q\}} \big( \mu_j(P)\log\mu_j(P) }+ \frac1m \big| \mu_j(Q)\log\mu_j(Q) - \nu(Q)\log\nu(Q)\big|
\\
& \quad \leq \frac{\eps}{2} + \frac{2e^{-1}}{m} < \eps. 
\end{align*}
This and Proposition~\ref{commonpartition} yield
\begin{align*}
\nu(\homsp) h_{\frac{\nu}{\nu(\homsp)}}(T) + 3\eps &> \frac1m H_{\mu_j}(\eta_0^{m-1}) \geq h_{\mu_j}(T,\eta) 
\\
& > h_{\mu_j}(T) - \tfrac1s - \tfrac12 h_m(T) \cdot\big( 1- \mu_j(\homsp_{\leq s})\big).
\end{align*}
Hence
\[
 \nu(\homsp) h_{\frac{\nu}{\nu(\homsp)}}(T) + \tfrac12 h_m(T) \cdot \big( 1- \nu(\homsp_{\leq s})\big) + 3\eps + \tfrac1s \geq \limsup_{j\to\infty} h_{\mu_j}(T).
\]
Letting $\eps$ tend to $0$ and $s$ tend to infinity, it follows
\[
 \nu(\homsp) h_{\frac{\nu}{\nu(\homsp)}}(T) + \tfrac12 h_m(T) \cdot\big( 1- \nu(\homsp)\big) \geq \limsup_{j\to\infty} h_{\mu_j}(T).
\]
\end{proof}

As an immediate consequence of Proposition~\ref{maxentropy} and Theorem~\ref{entropycusp} we obtain the corollaries stated in the introduction.

\begin{cor}\label{entropcor}
Let $(\mu_j)_{j\in\N}$ be a sequence of $T$-invariant probability measures on $\homsp$ such that $\liminf_{j\to\infty} h_{\mu_j}(T) \geq c$. Let $\nu$ be any weak* limit point of $(\mu_j)$. Then 
\[
 \nu(\homsp) \geq \frac{2c}{h_m(T)} -1.
\]
Moreover, if 
\[
 \nu(\homsp) = \frac{2c}{h_m(T)} -1 > 0,
\]
then $h_{\frac{\nu}{\nu(\homsp)}}(T) = h_m(T)$ and $\frac{\nu}{\nu(\homsp)}$ is the Haar measure on $\homsp$.
\end{cor}

\section{Hausdorff dimension of orbits missing a fixed open subset}\label{notfull}

In this section we prove Theorem~B from the introduction, which is an application of Theorem~\ref{entropycusp} and the methods for its proof and answers a question by Barak Weiss about the Hausdorff dimension of the set of all orbits which miss a fixed open subset of $\homsp$. We note that for a \textit{compact} quotient this is a simple corollary of semi-continuity of entropy and uniqueness of the measure of maximal entropy. In the presence of cusps, the methods of this paper become relevant. We also note that related results have been obtained by Shi \cite{Shi} but to our knowledge these do not provide the following results as corollaries.

Let $\open \subseteq \homsp$ be a non-empty open subset. Let $\ok$ denote the set of points in $\homsp$ whose forward-$A$-orbits do not intersect $\open$, that is
\[
 \ok \sceq \{ x\in\homsp \mid \forall\, t\geq 0\colon xa_t \notin\open\}.
\]
In the following we will show that $\ok$ cannot have full Hausdorff dimension as claimed in Theorem B of the introduction. Instead of Theorem B we will prove a (stronger) discretized version. To that end we now define $T$ 
using $\tilde a=a_e$, ($e = \exp(1)$), so that 
\[
 T \colon \homsp \to \homsp,\quad x\mapsto xa_e,
\]
denotes the time-one (discrete) geodesic flow. Note that then the maximal entropy of $T$ is
\[
 h_m(T) = \frac{p_1}2 + p_2.
\]
We consider the set 
\[
 \ok' \sceq \{ x\in \homsp \mid \forall\, n\in\N_0\colon T^nx \notin \open\}.
\]
Then Theorem B is implied by the following

\begin{thm}\label{thmdiscrete}
The Hausdorff dimension of~$\ok'$ satisfies  $\dim_H \ok' < \dim\homsp=\dim G$.
\end{thm}

For convenience we recall the following definitions, adapted to our current set-up. For $r>0$, the open ball in $G$ centered at the identity of $G$ with radius $r$ is denoted by $B_r^G$. For $L\in\N$, the Bowen $L$-ball in $G$ with radius parameter $r$ is 
\[
 B_L = B_L(r) = \bigcap_{j=0}^{L-1} a_e^j B_r^G a_e^{-j}.
\]
Finally, for each $x\in\homsp$, the Bowen $L$-ball in $\homsp$ with center $x$ is
\[
 x B_L = x B_L(r).
\]

\textit{Strategy for the proof of Theorem~\ref{thmdiscrete}:}
We cover $\ok'$ by countably many (small) bounded open sets, say by $A(n), n\in\N$,  and estimate the Hausdorff dimension of each of the sets
\[
 \mc W_n \sceq \ok'\cap A(n).
\]
By countable stability of Hausdorff dimension we have
\[
 \dim_H \ok' = \sup\{ \dim_H\mc W_n \mid n\in\N\}.
\]
Thus, we have to show that 
\[
 \dim_H\mc W_n < \dim\homsp - \eps_0=\dim G-\eps_0
\]
for some $\eps_0>0$ not depending on $n\in\N$. To seek a contradiction we assume that 
for (some small, to be determined below) $\eps_0>0$ we find a set $\mc W = \mc W(\eps_0)$ among the sets $\mc W_n$ such that 
\begin{equation}\label{Wdim}
 d \sceq \dim_H \mc W \geq \dim \homsp - \eps_0.
\end{equation}
Frostman's Lemma assures the existence of a probability measure $\mu$ on $\mc W$ such that 
\begin{equation}\label{frostmeasure}
 \mu(xB^G_r) \ll r^{d-\eps_0}\leq r^{\dim G-2\eps_0}
\end{equation}
for any $x\in \homsp$ and any $r\in(0,1]$, with an implied constant only depending on $\mu$. Then we will give an upper bound for the number of Bowen $L$-balls needed to cover $\mc W$ as well as the $\mu$-mass of an Bowen $L$-ball. Bounding the $\mu$-mass of $\mc W$ (which is $1$) via these Bowen balls will result in a contradiction.

We start by choosing a good value for $\eps_0$. Recall that $\mc M_1(\homsp)^T$ denotes the space of $T$-invariant probability measures on $\homsp$ and that $h_m(T)$ is the maximal entropy of $T$ (see Proposition~\ref{maxentropy}).

\begin{lemma}\label{univgap}
There exists $\delta_0>0$ such that 
$h_\nu(T)\leq h_m(T)-\delta_0$ for all~$\nu\in \mc M_1(\homsp)^T$ with~$\supp\nu\subseteq\ok'$. 
\end{lemma}

\begin{proof}
To seek a contradiction assume that there exists a sequence $(\nu_n)_{n\in\N}$ in $\mc M_1(\homsp)^T$ with $\supp \nu_n \subseteq \ok'$ for each $n\in\N$ such that 
\[
 h_{\nu_n}(T) \geq h_m(T) - \frac1n.
\]
Let $\nu$ be any weak* limit point of $(\nu_n)$. Then $\supp\nu \subseteq \ok'$ as $\ok'$ is closed. Since $\lim h_{\nu_n}(T) = h_m(T)$, Corollary~\ref{entropcor} yields $\nu(\homsp) = 1$ and $h_{\nu}(T) = h_m(T)$. Thus, $\nu$ is the Haar measure on $\homsp$ and hence $\supp\nu = \homsp$. This is a contradiction.
\end{proof}

We fix $\delta_0 \in (0, h_m(T)/4]$ with the properties as in Lemma~\ref{univgap} (the upper bound will be needed for Proposition~\ref{boundnu} below), set
\[
 \eps_0 \sceq \frac{\delta_0}{20},
\]
and assume the existence of~$\mc W\subset\ok'$ satisfying~\eqref{Wdim}. We also fix a probability measure $\mu$ on $\mc W$ satisfying \eqref{frostmeasure}. 

We pick a weak* limit point $\nu$ of 
\[
 \frac{1}{L} \sum_{j=0}^{L-1} T^j_*\mu \qquad\text{as $L\to\infty$.}
\]
Then $\nu$ is $T$-invariant and $\nu(\open) = 0$. We will see in Proposition~\ref{boundnu} below that $\nu(\homsp)>0$. Then Lemma~\ref{univgap} shows
\[
 h_{\frac{\nu}{\nu(\homsp)}}(T) \leq h_m(T) -\delta_0,
\]
which will allow us to establish nontrivial bounds on the number of Bowen $L$-balls needed to cover $\mc W$.

We start by deriving a lower bound for $\nu(\homsp)$, where the following will be needed.

\begin{lemma}\label{boxmass}
For any sufficiently (depending only on~$G$) small $r>0$, any $x\in \homsp$  and any $L\in\N$  we have
\[
 \mu(x \overline{B_L(r)})\leq c_\mu r^{\dim G-2\eps_0} e^{\left( -h_m(T) + 2\eps_0\right)L},
\]
where the constant $c_\mu$ depends on the implied constant in \eqref{frostmeasure} and on~$G$.
\end{lemma}

\begin{proof}
Choose a maximal collection of elements~$g_1,\ldots,g_q\in \overline{B_L(r)}$
for which the balls $g_1 B_{re^{-L}},\ldots,g_q B_{r e^{-L}}$ are pairwise disjoint. 
Note that $B_{re^{-L}}\subseteq B_L(r)$
and so~$g_j B_{re^{-L}}\subseteq B_L(2r)$. By Lemma~\ref{massBowen} the Haar measure of~$B_L(r)$
is between bounded multiples (depending on~$G$) of~$r^{\dim G}e^{-h_m(T)L}$ and the Haar measure of~$B_{re^{-L}}$
is between bounded multiples of~$r^{\dim G}e^{-L\dim G}$. Hence it follows that~$q\leq c 'e^{L(\dim G-h_m(T))}$
for some constant~$c'$ that only depends on~$G$.

The maximality of $q$ implies $x\overline{B_L(r)}\subseteq xg_1B_{2re^{-L}}\cup\cdots\cup x g_qB_{2re^{-L}}$. We now apply \eqref{frostmeasure} and obtain
\begin{align*}
 \mu(x\overline{B_L(r)})&\leq 
  q c_\mu' 2^{\dim G-2\eps_0} r^{\dim G-2\eps_0} e^{-L(\dim G-2\eps_0)}\\
&\leq  c_\mu r^{\dim G-2\eps_0} e^{(\dim G-h_m(T) -\dim G +2\eps_0)L}   ,
\end{align*}
where~$c_\mu'$ is the implied constant in~\eqref{frostmeasure}.
\end{proof}

For a subset $V\subseteq [0,L-1]$ we set
\[
 Q_{s,V} \sceq \left\{ x\in \homsp \left\vert\  \forall\, j\in [0,L-1] \colon \left(T^jx \in \homsp_{> s} \Leftrightarrow j\in V\right) \right.\right\}.
\]

\begin{lemma}\label{numberpartition}
Let $s > s_3$ and 
\[
L \geq 2\log\left(\frac{s}{s_3}\right)+1.
\]
Then there are at most $ e^{f(s)   L  }$
subsets $V\subseteq [0,L-1]$ for which $Q_{s,V}$ is nonempty, where 
$$ f(s) \sceq \frac{4 \log\left( 2 \log\left(\frac{s}{s_3}\right) + 2 \right)}{\log\left(\frac{s}{s_3}\right)}.$$
\end{lemma}

\begin{proof}
Let $V \subseteq [0,L-1]$ and decompose $V$ as in \eqref{decompV} (using $s'\sceq s$). We will show that $Q_{s,V}$ being nonempty implies that there is a uniform nontrivial minimal distance between $k_n+K_n$ and $k_{n+1}$ since a trajectory going down from $\homsp_{>s}$ cannot go back up before entering $\homsp_{\leq s_3}$ (see Lemma~\ref{s2s3}\eqref{descent}). The existence of this distance yields restrictions on those $V$ for which $Q_{s,V}\not=\emptyset$.

At first suppose that we have $x\in \homsp$ with $\height(x) \leq s_3$ and $\height(T^j x) > s$ for some $j\in \N$. By Lemma~\ref{s2s3}\eqref{below} we may suppose  $\height(xa_t) > 2s_1$ for all $t\in [1,e^j]$. We aim to prove a nontrivial lower bound on $j$.  By Proposition~\ref{ht_super}  there exists a unique cusp representative $\xi\in\Xi$ and an element $g\in G$ with $x=\Gamma g$ such that
\[
 \height(xa_t) = \height_\xi(xa_t) = \left( \frac{ \| v_\xi \varrho(ga_t) \| }{ \| v_\xi \varrho(\xi) \| }\right)^{-\frac1q}
\]
for all $t\in [1,e^j]$. Proposition~\ref{sojourn} implies $g \in \xi NAMU$, say $g=\xi n a_r m u$ with $u=\sigma(1, Z, X)\sigma$. Lemma~\ref{value} shows
\[
 \height_\xi(xa_t) = r \cdot \frac{\frac1t}{ \left( \frac1t + \frac14|X|^2 \right)^2 + |Z|^2 }
\]
for all $t\in [1,e^j]$. In particular,
\[
 s_3 > \height_\xi(x) = \frac{r}{\left( 1 + \frac14 |X|^2 \right)^2 + |Z|^2}
\]
and
\[
 \height(xa_{e^j}) = r \cdot \frac{ e^{-j} }{ \left( e^{-j} + \frac14|X|^2 \right)^2 + |Z|^2 } > s.
\]
Therefore
\[
 \frac{ \left(1+\frac14|X|^2\right)^2 + |Z|^2  }{ \left(e^{-j} + \tfrac14|X|^2\right)^2 + |Z|^2 } \cdot e^{-j} > \frac{s}{s_3}.
\]
Together with the elementary estimate
\[
 e^{2j} \geq \frac{ \left(1+\frac14|X|^2\right)^2 + |Z|^2  }{ \left(e^{-j} + \tfrac14|X|^2\right)^2 + |Z|^2 }
\]
it follows that 
\[
 j > \log\left( \frac{s}{s_3}\right).
\]
Suppose now that we have $x\in\homsp$ with $\height(x) > s$ and $\height(T^jx) \leq s_3$ for some $j\in\N$. Invoking Lemma~\ref{s2s3}\eqref{above}, we can deduce as before that
\[
 j > \log\left(\frac{s}{s_3}\right).
\]
We set
\[
 j_0 \sceq \left\lceil  \log\left(\frac{s}{s_3}\right) \right\rceil.
\]
Lemma~\ref{s2s3}\eqref{descent} implies that 
\[
 k_n + K_n + 2j_0 \leq k_{n+1}
\]
for $n=1,\ldots, m_1-1$. 
Let
\[
 Q_0^L(s) \sceq \bigvee_{j=0}^{L-1} T^{-j}\{\homsp_{\leq s}, \homsp_{>s}\}.
\]
If $L=2j_0-1$, the definition of $j_0$ yields that the cardinality of $Q_0^L(s)$ is (note $j_0\geq 1$)
\[
 \leq 1 + {2j_0 \choose 2} = \frac{(2j_0)^2}{2} + 1 - j_0 \leq (2j_0)^2.
\]
For an arbitrary $L$ the set $[0,L-1]$ is covered by the disjoint union 
\[
 [0,L-1] \subseteq \bigcup_{h=0}^{k_L} h\cdot(2 j_0 -1) + [0,2j_0-2]
\]
with
\[
 k_L \sceq \left\lceil \frac{L}{2j_0 -1} \right\rceil.
\]
For each $h\in\{0,\ldots, k_L\}$, the cardinality of 
\[
 \big\{ V \cap \left( h\cdot(2 j_0 -1) + [0,2j_0-2] \right) \ \big\vert\  V\subseteq [0,L-1],\ Q_{s,V} \not=\emptyset \big\}
\]
is at most $(2j_0)^2$. Therefore, there are at most 
\[
 (2j_0)^{2k_L}
\]
subsets $V\subseteq [0,L-1]$ with $Q_{s,V} \not= \emptyset$. Hence the cardinality of $Q_0^L(s)$ is bounded from above by
\[
 \exp\left( 2k_L\cdot \log (2j_0) \right).
\]
Using
\[
 k_L \leq \frac{L}{2 j_0 -1} + 1
\]
and
\[
 \frac{ \log\left( 2 \left\lceil \log\left(\frac{s}{s_3}\right) \right\rceil\right) }{ 2\left\lceil \log\left(\frac{s}{s_3}\right)\right\rceil - 1} \leq \frac{1}{\log \left(\frac{s}{s_3}\right)} \log\left( 2 \log\left(\frac{s}{s_3}\right) + 2\right),
\]
the statement of the proposition follows easily.
\end{proof}

For $L\in\N$, a subset $V\subseteq [0,L-1]$ and $s>0$ we set
\[
 Z_{s,L}(V) \sceq \left\{ x\in\mc W\cap \homsp_{\leq s} \left\vert\  \forall\, j\in [0,L-1]\colon \left(T^jx \in \homsp_{>s} \Leftrightarrow j\in V\right)\right.\right\}.
\]
Lemma~\ref{numberpartition} immediately provides an upper bound on the number of nonempty sets $Z_{s,L}(V)$. By increasing $s_3$ we may assume from now on that $\mc W \subseteq \homsp_{\leq s_3}$.

\begin{lemma}\label{numberZbox}
Suppose that the parameter $r$ in the definition of Bowen balls is an injectivity radius of $\homsp_{\leq s_3}$. Let $s>s_3$, $L$ as in Lemma~\ref{numberpartition} and $V\subseteq [0,L-1]$. Then the set $Z_{s,L}(V)$ can be covered with
\[
 c_{r,\mc W}e^{c(s) L + h_m(T) (L-\frac12|V|)}
\]
Bowen $L$-balls, where $c(s) \to 0$ as $s\to \infty$ and the constant $c_{r,\mc W}$ does not depend on $s,L$ and $V$.
\end{lemma}

\begin{proof}
The proof is similar to that of Lemma~\ref{coverimp} with $s=s'$.
 We can cover $Z_{s,L}(V)$ with finitely many balls $xB_r$ with $x\in Z_{s,L}(V)$, say
\begin{equation}\label{coverZbox}
 Z_{s,L}(V) \subseteq \bigcup_{j=1}^{c_{r,\mc W}} x_jB_r.
\end{equation}
The necessary number $c_{r,\mc W}$ of such balls is bounded 
by a constant independent of $s,L$ and $V$. Suppose now that $x_0B_r$ is one of the sets used in the covering \eqref{coverZbox} and consider
\[
 \mc Z \sceq Z_{s,L}(V) \cap x_0B_r.
\]
If $\mc Z$ is covered with say $d_1$ Bowen $\ell$-balls with center in $\mc Z$, then, as in Lemma~\ref{coverimp} ``step corresponding to adding $\tilde V_j$'', an excursion into $\homsp_{>s}$ of length $\ell_1$ starting at~$\ell+1$ has the effect that $\mc Z$ is covered by
\[
 {c} d_1 e^{\frac12 h_m(T) \ell_1}
\]
Bowen $(\ell+\ell_1)$-balls. In contrast (as in Lemma~\ref{coverimp} ``step corresponding to adding $W_j$''), a stay  in $\homsp_{\leq s}$ of length $\ell_2$ only leads to a trivial estimate, that is $\mc Z$ is covered by
\[
 {c} d_1 e^{h_m(T)\ell_2}
\]
Bowen $(\ell+\ell_2)$-balls.

By iteratively applying these two steps, we see that $\mc Z$ can be covered with
\[
 c^{{2} m} e^{h_m(T) \left( L - \frac12 |V|\right)}
\]
Bowen $L$-balls, where $c$ is a constant independent of $s,L$ and $V$, and $m$ is the number of excursions into  $\homsp_{>s}$. The proof of Lemma~\ref{numberpartition} now shows that 
\[
 m\leq \frac{2L}{\log\left(\frac{s}{s_3}\right)} + 1.
\]
This completes the proof.
\end{proof}

\begin{prop}\label{boundnu}
We have
\[
 \nu(\homsp) \geq 1- \frac{4\eps_0}{h_m(T)}\geq 1-\frac1{20}.
\]
\end{prop}

\begin{proof}
For $L\in\N$ set 
\begin{equation}\label{muL}
 \mu_L \sceq \frac1L\sum_{j=0}^{L-1} T^{j}_*\mu.
\end{equation}
Then $\mu_L$ converges along some subsequence to $\nu$ in the weak* topology. For any $s>s_3$ we have
\[
\mu_L(\homsp_{>s})  = \frac1L \sum_{j=0}^{L-1} \mu(T^{-j}\homsp_{>s})
 = \frac1L \sum_{j=0}^{L-1} \mu(\homsp_{\leq s}\cap T^{-j}\homsp_{>s}).
\]
since~$\mc W\subset\homsp_{\leq s_3}$. For $x\in\homsp$ we set
\[
 V_x \sceq \big\{ j\in [0,L-1] \ \big\vert\ T^jx\in\homsp_{>s}\big\}.
\]
In the following let $\varrho>0$ be a constant depending on $s$, to be fixed below. Then, it follows
\begin{align*}
\frac1L&\sum_{j=0}^{L-1} \mu(\homsp_{\leq s} \cap T^{-j}\homsp_{>s}) = \frac1L \sum_{n=1}^L n\mu(\{x\in\homsp_{\leq s} \mid |V_x| = n\} )
\\
& = \frac1L \sum_{n=1}^{\lceil \varrho L\rceil-1} n \mu(\{x\in\homsp_{\leq s} \mid |V_x| = n\} ) + \frac1L \sum_{n=\lceil\varrho L\rceil}^L n \mu(\{x\in\homsp_{\leq s} \mid |V_x| = n\} )
\\
& \leq \frac1L \left( \lceil\varrho L\rceil -1\right) \mu(\homsp_{\leq s}) + \frac1L\cdot L \cdot \mu(\{x\in\homsp_{\leq s} \mid |V_x| \geq \varrho L \} )
\\
& \le \varrho + \mu(\{x\in\homsp_{\leq s} \mid |V_x| \geq \varrho L \} ).
\end{align*}
Using Lemmas~\ref{numberpartition} and \ref{numberZbox} for sufficiently large $L$ allows us to estimate the latter term. 
Combining this with the estimate in Lemma~\ref{boxmass} we get
\begin{align*}
\mu(\{x&\in\homsp_{\leq s} \mid |V_x| \geq \varrho L \} ) \leq \mu\left( \bigcup_{|V| \geq \varrho L} Z_{s,L}(V) \right)
\\
& \leq c_{r,\mc W} e^{f(s) L} e^{c(s)L + h_m(T) (L-\frac12\varrho L)} \cdot 
c_\mu r^{\dim G-2\eps_0} e^{\left( -h_m(T) + 2\eps_0 \right)L}
\\
& \leq c'e^{(\wt c(s) + 2\eps_0 - \frac12h_m(T)\varrho)L},
\end{align*}
where $c'$ is a constant depending on $r$ but not on $s$ or $L$,  and $\wt c(s) \sceq f(s) + c(s)$ tends to $0$ as $s\to\infty$.
If we choose~$\varrho$ so that the exponent is negative, i.e.\ with 
\[
\varrho=\varrho(s) > \frac{2\wt c(s)+4\eps_0}{h_m(T)},
\]
then 
\[
 \eps(L)=\mu(\{x\in\homsp_{\leq s} \mid |V_x| \geq \varrho(s) L \} ) \to 0 \quad\text{as $L\to \infty$}.
\]
Therefore, for sufficiently large $s$ and $L$, we have
\[
 \mu_L(\homsp_{> s}) \leq \varrho(s) + \eps(L).
\]
In turn,
\[
 \mu_L(\homsp_{\leq s}) \geq 1- (\varrho(s) + \eps(L)).
\]
Letting~$L\to\infty$ along the subsequence that gives~$\nu$ as the limit we obtain
\[
 \nu(\homsp_{\leq s}) \geq 1- \varrho(s),
\]
and
\[
 \nu(\homsp) \geq 1-\varrho\quad \text{for all $\varrho>\frac{ 4 \eps_0}{h_m(T)}.$}
\]
This proves the claim if one recalls that~$\delta_0\leq\frac{h_m(T)}{4}$ and~$\eps_0=\frac{\delta_0}{20}$.
\end{proof}

Our next goal is to use~$h_{\frac{\nu}{\nu(\homsp)}}(T)\leq h_m(T)-\delta_0$ to
give an upper bound on the number of Bowen $L$-balls needed to cover $\mc W$.
Let $\mc E_1(\homsp)^T$ denote the space of $T$-invariant ergodic probability measures on $\homsp$.

\begin{prop}\label{effectcover}
Let $\eps>0$. There exist $L_1\in\N$ and $Y \subseteq \homsp$ with
$\nu(Y)>\nu(\homsp)-\eps$ such that for all $L\geq  L_1$, the set $Y$ can be covered with
\[
 e^{(h_m(T)-\delta_0 + \eps)L}
\]
Bowen $L$-balls (with centers in $Y$). Here we may use a radius parameter~$r=r(\eps)$ in the definition
of the Bowen balls such that~$10er$ is an injectivity radius on the compact set~$\bar{Y}$.  
\end{prop}

\begin{proof}
We normalize the measure $\nu$ to 
\[
 \sigma\sceq \frac{\nu}{\nu(\homsp)}.
\]
By the ergodic decomposition of $\sigma$ and since $\sigma(\open)=0$ we find a
subset $\homsp'\subseteq \ok'$ with $\sigma(\homsp') = 1$ and a measurable map $\homsp'\to \mc E_1(\homsp)^T$, $x\mapsto \sigma_x$, such that
\[
 \sigma = \int_{\homsp'} \sigma_x\, d\sigma(x)
\]
and $\sigma_x(\ok') = 1$ for all $x\in\homsp'$. Let $\mc P$ be any countable partition of $\homsp$ with finite partition entropy $H_\sigma(\mc P) < \infty$. For any $n\in\N$ and any $x\in\homsp$ let $[x]_{\mc P_0^{n-1}}$ denote the partition element in $\mc P_0^{n-1}$ which contains $x$. 

We want to find a lower estimate of $\sigma([x]_{\mc P_0^{n-1}})$. For this, let 
\[
 I_\sigma(\mc P_0^{n-1})(x) \sceq - \log \sigma\left( [x]_{\mc P_0^{n-1}}\right)
\]
denote the information function of $\mc P_0^{n-1}$. By the Shannon-McMillan-Breiman Theorem (see e.g.~\cite[Theorem~3.2]{ELW_book}) there exists a subset $\homsp''\subseteq\homsp'$ with $\sigma(\homsp'') = 1$ such that 
\[
 \frac1n I_\sigma(\mc P_0^{n-1})(x) \to h_{\sigma_x}(T,\mc P) \quad \text{as $n\to\infty$}
\]
for all $x\in\homsp''$ and in $L^1$. 

Fix some $\eps>0$. By the above there 
exists a subset $Y_1 \subseteq \homsp''$ with $\sigma(Y_1)>1-\eps$ and $L_1\in\N$ such that for all $L\geq L_1$ and all $x\in Y_1$ we have
\[
 \frac{1}{L}I_\sigma\left(\mc P_0^{L-1}\right)(x) < h_{\sigma_x}(T,\mc P) + \eps.
\]
From Lemma~\ref{univgap} (note that $\supp\sigma_x \subseteq \ok'$ for $x\in \homsp'$) and the definition of the dynamical entropy it follows that
\[
 h_{\sigma_x}(T,\mc P)\leq h_{\sigma_x}(T) \leq h_m(T) -\delta_0
\]
for each $x\in\homsp'$. Thus, for any $L\geq L_1$ and $x\in Y_1$ we have
\[
 -\frac1L \log\sigma\left( [x]_{\mc P_0^{L-1}}\right) < h_m(T)-\delta_0 + \eps.
\]
Hence
\begin{equation}\label{est1}
 \sigma\left( [x]_{\mc P_0^{L-1}} \right) \geq e^{-L\left(h_m(T)-\delta_0+\eps\right)}.
\end{equation}

We now pick a more specific partition as follows. A slight modification of the proof of \cite[7.53]{Einsiedler_Lindenstrauss} (we provide more details in Section~\ref{modifications}) allows us to choose a relatively compact subset $Q\subseteq \homsp$ with $\sigma(Q)>1-\eps$ such that
there exists a countable partition $\mc P$ of $\homsp$ with $H_\sigma(\mc P)<\infty$ such that\footnote{The 
construction of the partition $\mc P$ takes into account any previously
fixed injectivity radius~$r>0$ on~$Q$, and the Bowen balls in \eqref{est2} will use this as the radius parameter.} 
\begin{equation}\label{est2}
 [x]_{\mc P_0^{L-1}} \subseteq xB_L
\end{equation}
for all $x\in Q$ and all $L\in \N$. 

Set $Y\sceq Y_1\cap Q$. Then $\sigma(Y)>1-2\eps$. Let $L\geq L_1$. 
For any~$x\in Y$ the formula \eqref{est1} gives a lower bound
on the $\sigma$-measure of the partition element that contains~$x$. It follows that we get the upper bound~$e^{L(h_m(T)-\delta_0+\eps)}$
for the number of distinct (and hence disjoint) elements that intersect~$Y$ nontrivally.
For each of those we pick an element in~$Y$ to use as the center for a Bowen~$L$-ball, which then covers the associated
partition element by \eqref{est2}. These Bowen~$L$-balls now cover~$Y$ as required.
\end{proof}

\begin{lemma}\label{manyinY}
There exist $Y\subseteq \homsp$ and $L_1\in \N$ with the following properties.
Let  $L_0\geq L_1$ and choose a cover by Bowen~$L_0$-balls
as in Proposition~\ref{effectcover}. Denote the closure of the union of the cover by~$Y_0$.
Then there exist infinitely many $K$ with the property that
\begin{equation}\label{eq:manyinY}
	\frac1K\sum_{k=0}^{K-1} T_*^{kL_0}\mu(Y_0a_e^{-n}) > 1-\frac{ 5\eps_0}{h_m(T)}
\end{equation}
for some $n\in\{0,\ldots, L_0-1\}$. 
\end{lemma}

\begin{proof}
Let $\eps\in (0,1)$ (determined below).
Let~$(Y,L_1)$ be as in Proposition~\ref{effectcover}.
Let $L_0\geq L_1$ and define~$Y_0$
as in the statement of the lemma. As~$Y_0$ is closed and by definition of $\nu$ we have
\[
\liminf \frac1L \sum_{j=0}^{L-1}T^j_*\mu(Y_0) \geq \nu(Y_0),
\]
where $\liminf$ is taken over a subsequence of $L$'s such that $(\mu_L)$ converges to $\nu$ (cf.\@ the definition of $\nu$ and \eqref{muL}).
Since $Y_0\supseteq Y$ we have $\nu(Y_0)\geq\nu(Y) > \nu(\homsp) - \eps$. Hence there exist infinitely many~$L$ with
\[
 \frac1L\sum_{j=0}^{L-1}T^j_*\mu(Y_0) > \nu(\homsp)-2\eps.
\]
We divide~$L$ by~$L_0$ with remainder, write~$K=\lceil\frac{L}{L_0}\rceil$ and obtain 
\begin{align*}
\frac1L \sum_{j=0}^{L-1}T^j_*\mu(Y_0) & = \frac1L \left( \sum_{n=0}^{L_0-1} \sum_{k=0}^{K-1} T_*^{kL_0+n}\mu(Y_0) + \sum_{j=L_0K}^{L-1}T^j_*\mu(Y_0)\right)
\\
& =  \frac{K}{L} \sum_{n=0}^{L_0-1}\cdot\frac{1}{K} \sum_{k=0}^{K-1} T^{kL_0+n}_*\mu(Y_0) + \frac1L \sum_{j=L_0K}^{L-1} T^j_*\mu(Y_0).
\end{align*}
The last sum has at most $L_0$ summands and so converges to~$0$ as~$L\to\infty$. 
Moreover,~$\frac{K}L$ converges to~$\frac1{L_0}$. 
Therefore, for large enough~$L$ from our sequence we must have
\[
 \nu(\homsp)-3\eps < \frac1{L_0}\sum_{n=0}^{L_0-1} \frac1K\sum_{k=0}^{K-1} T^{kL_0+n}_*\mu(Y_0).
\]
Hence we find $n\in\{0,\ldots, L_0-1\}$ such that 
\[
 \frac1K\sum_{k=0}^{K-1} T^{kL_0+n}_*\mu(Y_0) > \nu(\homsp)-3\eps.
\]
Proposition~\ref{boundnu} now gives
\[
\frac1K\sum_{k=0}^{K-1} T_*^{kL_0}\mu(Y_0a_e^{-n}) > 1-\frac{ 4\eps_0}{h_m(T)}-3\eps.
\]
Choosing~$\eps>0$ sufficiently small, the lemma follows.
\end{proof}

\begin{prop}\label{counting}
Assuming that~$L_0$ in Lemma~\ref{manyinY} is sufficiently big there exist infinitely many~$K$  such that a~$\mu$-proportion of at least~$1/2$ of the set~$\mc W$  can be covered with
\[
e^{(h_m(T)-\frac78\delta_0+2\eps_0)KL_0}
\]
Bowen $L_0K$-balls. 
\end{prop}

\begin{proof}
Let~$\eps=\eps_0>0$ and apply Proposition~\ref{effectcover} to obtain~$L_1$. Let~$L_0\geq L_1$ (below we will
give another constraint for~$L_0$) and let~$Y_0$ be the closure of the union of the Bowen~$L_0$-balls covering~$Y$ as in Lemma~\ref{manyinY}. Fixing~$L_0$ there are infinitely many~$K$ for which the conclusion~\eqref{eq:manyinY} of Lemma~\ref{manyinY} holds. Fix one such~$K$ and the corresponding~$n$.
For a given~$y\in \mc W$ we define
\[
V(y)\sceq \{k\in[0,K)\mid T^{n+kL_0}(y)\in Y_0\}.
\]
We fix one  subset~$ V$ of~$[0,K)$ and define
\[
\mc Z (V) \sceq \{y\in\mc W\mid V(y)= V_0\}.
\] 
We claim that we can cover $\mc Z(V)$ with 
\begin{equation}\label{reallyneeded}
 c_{r,\mc W}b_1 e^{h_m(T)n} b_1^{K}e^{|V|(h_m(T)-\delta_0+\eps_0)L_0}e^{(K-|V|)h_m(T)L_0}
\end{equation}
Bowen~$KL_0$-balls, where $b_1$ is a constant independent of $K,L,V$ and $s$.

We may phrase the conclusion Lemma~\ref{manyinY} (in the case where~$n=0$) by saying that most 
points in~$\mc W$ spend a lot of time in relatively few Bowen~$L_0$-balls under the orbit w.r.t.~$T^{L_0}$.
We will use this together with a concatenation procedure to bound how many Bowen~$L_0K$-balls
are needed to cover a portion of~$\mc W$. To ensure that we always work with Bowen balls with the same radius
we have to multiply the number of possibilities with an extra factor of~$b_1$ at each concatenation step 
resulting  in the~$b_1^K$-factor (and need the~$b_1 e^{h_m(T)L_0}$-factor to handle the case where~$n\neq 0$).

More precisely, let~$b_1$ be an upper bound for the number of Bowen~$L$-balls ~$gB_L$ (with parameter~$r$) 
that are needed to cover~$B_L(2r)$ independent of~$L$. The existence of~$b_1$ follows using a maximal
collection of pairwise disjoint Bowen~$L$-balls~$gB_L(r/2)$ with~$g\in B_L(2r)$ and Lemma~\ref{massBowen}. 
The simple concatenation step is now as follows. If~$y_1 B_{L_1}$ and~$y_2 B_{L_2}$
are two Bowen balls such that $y_2\in Y_0$, then~$y_1 B_{L_1}\cap (y_2 \overline{B_{L_2}}a_e^{-L_1})$ 
can be covered by~$ b_1$ Bowen~$L_1+L_2$-balls. Indeed take
some~$y\in y_1 B_{L_1}\cap (y_2 \overline{B_{L_2}}a_e^{-L_1}) $ 
and we get 
\[
(y_1 B_{L_1}(r))\cap (y_2 \overline{B_{L_2}(r)}a_e^{-L_1})\subseteq y B_{L_1+L_2}(2r)
\]
and the claim follows from the definition of~$b_1$. For the inclusion in the last step we need to know that~$2er$ is an injectivity radius at~$y_2$, which we know as~$y_2\in Y_0$ is assumed.

We may also assume that $b_1e^{h_m(T)L_2}$ is an upper bound on the number of~$(L_1+L_2)$-Bowen balls that are needed to cover~$B_{L_1}$ (choose a pairwise disjoint collection of Bowen~$(L_1+L_2)$-balls~$gB_{L_0}(r/2)$ with~$g\in B_{L_1}$ and apply Lemma~\ref{massBowen}), and note that in the case where~$L_1=0$ we may define~$B_0$ to be the~$r$-ball around the identity.
The second possible step (which we will call an extension step) then concerns the case where the first Bowen ball~$y_1 B_{L_1}$ is given, but we have no additional information about~$y_1a_e^{L_1}$. In this case we cover $y_1B_{L_1}$
by~$b_1e^{h_m(T)L_2}$  Bowen~$(L_1+L_2)$-balls.

Using the above concatenation and extension steps iteratively we see that~$\mc Z(V_0)$ can be covered by \eqref{reallyneeded}-many Bowen $KL_0$-balls. Indeed we first cover~$\mc W$ by~$c_{r,\mc W}$
balls of radius~$r$ (or equivalently by Bowen~$0$-balls) and use the extension step with~$L_1=0$ and~$L_2=n$.
Starting with~$k=0$ and if~$k\in V$ we use the concatenation for~$L_1=n+kL_0$ and~$L_2=L_0$
using all possibilities for the Bowen~$L_0$-balls appearing in the definition of~$Y_0$ in Lemma~\ref{manyinY}, or use
the extension step if~$k\notin V$. For~$k=K$ we obtain Bowen~$n+KL_0$-balls and the claim follows.

Note that the above estimate is monotonically decreasing as~$|V|$ increases.
Fix some~$\alpha\in(0,1)$ and consider all subset $V\subseteq [0,K)$ with $|V|\geq K\alpha$. The above now gives that
\begin{equation}\label{niceset}
 \bigcup_{|V|\geq K\alpha}\mc Z(V)
\end{equation}
is covered by
\[
 c_{r,\mc W}b_1^{1+K} e^{h_m(T)L_0} e^{\alpha(h_m(T)-\delta_0+\eps_0)KL_0}e^{(1-\alpha)h_m(T)KL_0} \sum_{k=\lfloor\alpha K\rfloor}^{K}{K \choose k}
\]
Bowen~$KL_0$-balls. It is well know that Sterlings formula can be used to estimate the last sum. In fact,
using Sterlings formula one obtains~${K\choose k}=e^{K H(\frac{k}K)+o(K)}$ 
as~$K\to\infty$ independent of~$k\in[0,K]$ and where
\[
H(x) = -x\log x-(1-x)\log x.
\]
Note that there are at most~$K=e^{o(K)}$ summands.
Bounding $H$ by~$\log 2$ and taking the sum we obtain the upper bound
\begin{equation}\label{plussterling1}
	c_{r,\mc W}b_1^{1+K} e^{h_m(T)L_0} e^{\alpha(h_m(T)-\delta_0+\eps_0)KL_0}e^{(1-\alpha)h_m(T)KL_0}e^{K(\log 2+1)}
\end{equation}
if~$K$ is sufficiently large. Note that the product of the constant and the factors containing $K$ or $L_0$ but not both
can be bounded by~$e^{\eps_0 L_0K}$ if~$L_0$ and~$K$ are sufficiently big. 
We now choose~$\alpha=1-\frac{10\eps_0}{h_m(T)}$ and recall that~$\delta_0\leq\frac{h_m(T)}4$ and~$\eps_0=\frac{\delta_0}{20}$,
which gives~$\alpha\geq \frac78$.  Putting these estimates into~\eqref{plussterling1} we get
\begin{equation}\label{plussterling}
	 e^{\frac78(h_m(T)-\delta_0+\eps_0)KL_0}e^{\frac18h_m(T)KL_0}e^{\eps_0KL_0}\leq e^{(h_m(T)-\frac78\delta_0+2\eps_0)KL_0}
\end{equation}
for the upper bound on the number of Bowen~$KL_0$-balls.

We finally apply Lemma~\ref{manyinY}
to bound the set that is not covered by the Bowen balls in \eqref{plussterling}. In fact, using the above notation
we may rewrite \eqref{eq:manyinY} to get
\begin{align*}
 \sum_{k=0}^K \frac{k}K\mu\bigl(\bigl\{y\in\mc W\mid |V(y)|=k\bigr\}\bigr)&>1-\frac{5\eps_0}{h_m(T)}\intertext{or}
 \sum_{k=0}^K \bigl(1-\frac{k}K\bigr)\mu\bigl(\bigl\{y\in\mc W\mid |V(y)|=k\bigr\}\bigr)&<\frac{5\eps_0}{h_m(T)}.
\end{align*}
The latter implies
\begin{align*}
 (1-\alpha) \mu\bigl(\bigl\{y\in\mc W\mid |V(y)|<\alpha K \bigr\}\bigr)&<\frac{5\eps_0}{h_m(T)}\intertext{or}
 \mu\bigl(\bigl\{y\in\mc W\mid |V(y)|<\alpha K \bigr\}\bigr)&<\frac12.
\end{align*}
From this and since \eqref{plussterling} gives an upper bound for the number of Bowen~$KL_0$-balls
needed to cover~\eqref{niceset} we obtain the proposition.
\end{proof}

\begin{proof}[Proof of Theorem~\ref{thmdiscrete}]
Let~$L_0$ be sufficiently big for Proposition~\ref{counting} to hold. Together with Lemma~\ref{boxmass}
we obtain
\begin{align*}
 \frac12=\frac12\mu(\mc W) &\leq c_\mu r^{\dim G - 2\eps_0} e^{\left( -h_m(T) + 2\eps_0\right)KL_0}e^{(h_m(T)-\frac78\delta_0+2\eps_0)KL_0}
\\
&=\tilde c e^{(4\eps_0-\frac78\delta_0)KL_0}
\end{align*}
for some constant $\tilde c$ independent of $K$, and for infinitely many~$K$. Note that $4\eps_0-\frac78\delta_0<0$, which leads to a contradiction. This completes the proof.
\end{proof}

\section{Modification of the partition from \cite{Einsiedler_Lindenstrauss}}\label{modifications}

The partition $\mc P$ in the proof of Proposition~\ref{effectcover} is essentially identical with the partition in \cite[7.51]{Einsiedler_Lindenstrauss}. However our situation is slightly different to the one in \cite{Einsiedler_Lindenstrauss} for which reason we outline the necessary steps of proof. The differences are as follows:
\begin{itemize}
\item The measure $\sigma$ is not necessarily ergodic, and we cannot reduce to an ergodic situation as in \cite{Einsiedler_Lindenstrauss}.
\item We want to find a big set $Q$ on which the inclusion relation 
\[
 [x]_{\mc P_0^{N-1}} \subseteq x B_N
\]
holds for all $N\in\N$ (in \cite{Einsiedler_Lindenstrauss}, the mass of $Q$ does not matter as long as it is positive, and it is asked for the (weaker) relation
\[
 [x]_{\mc P_0^\infty} \subseteq x B_N
\]
for $N\in\N$ such that $xa^N\in Q$).
\item We do not need the lower bounds from \cite{Einsiedler_Lindenstrauss} on the atoms, which allows us to simplify $\mc P$ a bit.
\end{itemize}

The construction of $\mc P$ and the proofs of its properties proceeds in a number of steps.

\begin{enumerate}[1)]
\item Pick a subset $Q\subseteq\homsp$ which is open, relatively compact and which has mass $\sigma(Q) > 1-\eps$. Pick an injectivity radius $r$ of $Q$ and decompose $Q$ (up to measure zero) into finitely many subsets $Q_1,\ldots, Q_R$ with positive measure such that each of these subsets is contained in a Bowen ball with parameter $r/16$. Set
\[
 \mc Q \sceq \{ Q_1, \ldots, Q_R, \homsp\setminus Q\}.
\]
\item For each $i=1,\ldots, R$  and each $j\in\N$ let $Q_{ij}$ be the set of points $x\in Q_i$ which return to $Q$ with the $j$-th step but not earlier, that is
\[
 Q_{ij} = \{ x\in Q_i \mid xa_e^j\in Q,\ xa_e^\ell\notin Q\ \text{for $\ell=1,\ldots,j-1$}\}.
\]
Set
\[
 \wt{\mc Q}\sceq \{ \homsp\setminus Q,\ Q_{ij} \mid i=1,\ldots, R,\ j\in\N\}.
\]
\item We now decompose the partition elements $Q_{ij}$ into smaller subsets. For this we remark that each ball $x B^G_{r/16}$ can be covered with 
\[
 ce^{\kappa j}
\]
balls with parameter $e^{-j} r/8$ (here $c=c(G)$ is a constant and $\kappa= \dim G$ would work). Let $i\in\{1,\ldots, R\}$ and $j\in\N$.  Note that $Q_{ij} \subseteq x_i B^G_{r/16}$ for some $x_i\in Q$. We fix a cover $B_1,\ldots B_{N(j)}$ with $N(j)\leq ce^{\kappa j}$ by balls with parameter $e^{-j} r/8$. We define
\begin{align*}
Q_{ij1} &\sceq Q_{ij} \cap B_1
\\
Q_{ij2} &\sceq Q_{ij} \cap \big(B_2\setminus B_1\big)
\\
Q_{ij3} &\sceq Q_{ij} \cap \big(B_3\setminus\big(B_1\cup B_2\big)\big)
\end{align*}
and so on. Set
\[
\mc P \sceq \{  Q_{ijk},\ \homsp\setminus Q\mid i=1,\ldots,R,\ j\in\N,\ k=1,\ldots,N(j)\}.
\]
\end{enumerate}

\begin{lemma}
For $\sigma$-almost every $x\in Q$ we have
\[
 [x]_{\mc P_0^{N-1}} \subseteq xB_N
\]
for all $N\in\N$.
\end{lemma}

\begin{proof}
By the Poincar\'e Recurrence Theorem, $\sigma$-almost every point in $Q$ returns infinitely often to $Q$. We restrict to these point and pick such an $x$. Let $N\in \N$ and let 
\[
 y \in  [x]_{\mc P_0^{N-1}}.
\]
There exist $i\in\{1,\ldots, R\}$, $j\in\N$ and $k\in\{1,\ldots, N(j)\}$ such that $x, y\in Q_{ijk}$. Thus, there exists $g\in B^G_{e^{-j}r/4}$ such that $y=xg$. Then for any $\ell=0,\ldots, j$ we have
\[
 g \in a_e^\ell B^G_{e^{\ell-j}r/4} a_e^{-\ell} \subseteq a_e^\ell B_r^G a_e^{-\ell}.
\]
In the case $j \geq N-1$ it follows immediately
\[
 y \in xB_N.
\]
Suppose that $j < N-1$. Then we find $i_2,j_2,k_2$ such that $xa_e^j, ya_e^j \in Q_{i_2j_2k_2}$. Since
\[
 ya_e^j = xa_e^j \big(a_e^{-j}ga_e^j\big),
\]
and $a_e^{-j}ga_e^j \in B_r^G$ and $r$ is an injectivity radius, it follows that actually
\[
 a_e^{-j} g a_e^{j} \in B^G_{e^{-j_2}r/4}.
\]
Now reasoning inductively as above finally shows $y \in xB_N$.
\end{proof}

\begin{lemma}
The partition $\mc P$ has finite partition entropy $H_\sigma(\mc P)$. 
\end{lemma}

\begin{proof}
We note that\footnote{Actually there is equality if~$Q\supseteq \homsp_{\leq s_3}$, but this is not needed for the proof.}
\[
 \sum_{i,j} j \sigma(Q_{ij}) \leq 1.
\]
Using this, the proof is parallel to that in \cite{Einsiedler_Lindenstrauss}.
\end{proof}

\bibliographystyle{amsalpha}

\begin{thebibliography}{CDKR98}

\bibitem[BK83]{Brin_Katok}
M.~Brin and A.~Katok, \emph{On local entropy}, Geometric dynamics ({R}io de
  {J}aneiro, 1981), Lecture Notes in Math., vol. 1007, Springer, Berlin, 1983,
  pp.~30--38.

\bibitem[CDKR91]{CDKR1}
M.~Cowling, A.~Dooley, A.~Kor{\'a}nyi, and F.~Ricci, \emph{{$H$}-type groups
  and {I}wasawa decompositions}, Adv. Math. \textbf{87} (1991), no.~1, 1--41.

\bibitem[CDKR98]{CDKR2}
\bysame, \emph{An approach to symmetric spaces of rank one via groups of
  {H}eisenberg type}, J. Geom. Anal. \textbf{8} (1998), no.~2, 199--237.

\bibitem[Dan84]{Dani}
S.~G. Dani, \emph{On orbits of unipotent flows on homogeneous spaces}, Ergodic
  Theory Dynam. Systems \textbf{4} (1984), no.~1, 25--34.

\bibitem[EK12]{Einsiedler_Kadyrov}
M.~Einsiedler and S.~Kadyrov, \emph{Entropy and escape of mass for {${\rm
  SL}_3({\Bbb Z})\backslash{\rm SL}_3({\Bbb R})$}}, Israel J. Math.
  \textbf{190} (2012), 253--288.

\bibitem[EL10]{Einsiedler_Lindenstrauss}
M.~Einsiedler and E.~Lindenstrauss, \emph{Diagonal actions on locally
  homogeneous spaces}, Homogeneous flows, moduli spaces and arithmetic, Clay
  Math. Proc., vol.~10, Amer. Math. Soc., Providence, RI, 2010, pp.~155--241.

\bibitem[ELMV12]{ELMV}
M.~Einsiedler, E.~Lindenstrauss, Ph. Michel, and A.~Venkatesh, \emph{The
  distribution of closed geodesics on the modular surface, and {D}uke's
  theorem}, Enseign. Math. (2) \textbf{58} (2012), no.~3-4, 249--313.

\bibitem[ELW]{ELW_book}
M.~Einsiedler, E.~Lindenstrauss, and T.~Ward, \emph{Entropy in ergodic theory
  and homogeneous dynamics}, book project,
  http://maths.dur.ac.uk/$\sim$tpcc68/entropy/.

\bibitem[GR70]{Garland_Raghunathan}
H.~Garland and M.~S. Raghunathan, \emph{Fundamental domains for lattices in
  $\mathbb{R}$-rank {$1$} semisimple {L}ie groups}, Ann. of Math. (2)
  \textbf{92} (1970), 279--326.

\bibitem[{Hel}00]{Helgason_gga}
Sigurdur {Helgason}, \emph{{Groups and geometric analysis. Integral geometry,
  invariant differential operators, and spherical functions.}}, Providence, RI:
  American Mathematical Society (AMS), 2000 (English).

\bibitem[Hel01]{Helgason3}
Sigurdur Helgason, \emph{Differential geometry, {L}ie groups, and symmetric
  spaces}, Graduate Studies in Mathematics, vol.~34, American Mathematical
  Society, Providence, RI, 2001, Corrected reprint of the 1978 original.

\bibitem[HW13]{Hubert_Weiss}
P.~Hubert and B.~Weiss, \emph{Ergodicity for infinite periodic translation
  surfaces}, Compos. Math. \textbf{149} (2013), no.~8, 1364--1380.

\bibitem[Kad12]{Kadyrov}
S.~Kadyrov, \emph{Entropy and escape of mass for {H}ilbert modular spaces}, J.
  Lie Theory. \textbf{22} (2012), no.~3, 701--722.

\bibitem[KKLM]{KKLM}
S.~Kadyrov, D.~Y. Kleinbock, E.~Lindenstrauss, and G.~A. Margulis,
  \emph{Entropy in the cusp and singular systems of linear forms},
  arXiv.org:1407.5310.

\bibitem[KP]{Kadyrov_Pohl}
S.~Kadyrov and A.~Pohl, \emph{Amount of failure of upper-semicontinuity of
  entropy in noncompact rank one situations, and {H}ausdorff dimension},
  arXiv.org:1211.3019.

\bibitem[MT94]{Margulis_Tomanov}
G.~A. Margulis and G.~M. Tomanov, \emph{Invariant measures for actions of
  unipotent groups over local fields on homogeneous spaces}, Invent. Math.
  \textbf{116} (1994), no.~1-3, 347--392.

\bibitem[Poh10]{Pohl_isofunddom}
A.~Pohl, \emph{Ford fundamental domains in symmetric spaces of rank one}, Geom.
  Dedicata \textbf{147} (2010), 219--276.

\bibitem[Shi12]{Shi}
R.~Shi, \emph{Convergence of measures under diagonal actions on homogeneous
  spaces}, Adv. Math. \textbf{229} (2012), no.~3, 1417--1434.

\end{thebibliography}
\providecommand{\bysame}{\leavevmode\hbox to3em{\hrulefill}\thinspace}
\providecommand{\MR}{\relax\ifhmode\unskip\space\fi MR }
\providecommand{\MRhref}[2]{%
  \href{http://www.ams.org/mathscinet-getitem?mr=#1}{#2}
}
\providecommand{\href}[2]{#2}

\end{document}